\documentclass[11pt, leqno]{amsart}
\usepackage{amsmath, amssymb, latexsym}
\usepackage{mathrsfs}
\usepackage{enumerate}
\usepackage{color}
\usepackage[colorlinks=true, pdfstartview=FitV, linkcolor=blue,%
citecolor=blue, urlcolor=blue]{hyperref}
\usepackage[all, knot]{xy}
\usepackage{chngcntr}
\counterwithout{figure}{section}
\xyoption{arc}
\usepackage{tikz}
\usepackage{mathdots}
\usepackage{amsmath}
\newtheorem{thrm}{Theorem}

\allowdisplaybreaks

\numberwithin{equation}{section}

\newtheorem{theorem}{Theorem}[section]
\newtheorem{corollary}[theorem]{Corollary}
\newtheorem{lemma}[theorem]{Lemma}
\newtheorem{proposition}[theorem]{Proposition}

\theoremstyle{definition}
\newtheorem{definition}[theorem]{Definition}
\newtheorem{remark}[theorem]{Remark}

\newcommand{\mrm}{\mathrm}

\newcommand{\var}{\varepsilon}

\newcommand{\wit}{\widetilde}
\newcommand{\wih}{\widehat}

\newcommand{\mbf}{\mathbf}
\newcommand{\mbb}{\mathbb}
\newcommand{\mcal}{\mathcal}
\newcommand{\mak}{\mathfrak}
\newcommand{\smxylabel}[1]{{\text{\small$#1$}}}

\title[ Geometric approach to i-quantum
group of affine type $\mrm D$ ]
{Geometric approach to i-quantum
group of affine type $\mrm D$}

\author[Quanyong Chen]{Quanyong Chen}
\address{Harbin Engineering University,
Harbin, China}
\email{chenquanyong@hrbeu.edu.cn}

\author[Zhaobing Fan]{Zhaobing Fan}
\address{Harbin Engineering University,
Harbin, China}
\email{fanzhaobing@hrbeu.edu.cn}
\thanks{ }

\begin{document}
\begin{abstract}
  In this paper, we study the structures of Schur algebra and Lusztig algebra associated to partial flag varieties of affine type $\mrm {D}$. We show that there is a subalgebra of Lusztig algebra and the quantum groups arising from this subalgebras via stabilization procedures is a coideal subalgebra of quantum group of affine $\mak{sl}$ type.
  We construct monomial and canonical bases of the idempotented quantum algebra and establish the positivity properties of the canonical basis with respect to multiplication and the bilinear pairing.
\end{abstract}

\maketitle

\setcounter{tocdepth}{2}
\tableofcontents
\section{Introduction}

In 1990, Beilinson, Lusztig and MacPherson \cite{BLM90}  provided a geometric realization of quantum Schur algebra $\mbf S_{n,d}^{\rm fin}$ as convolution algebras on pairs of partial flags over a finite field.
By using multiplication formulas in $\mbf S_{n,d}^{\rm fin}$ with divided powers of Chevalley generators, they furhter realized the quantum group $\mbf U_q(\mak{gl}_n)$ in the projective limit of the quantum Schur algebras (as $d\rightarrow \infty$).
More importantly, an idempotented version of quantum group $\mbf U_q(\mak{gl}_n)$ is discovered inside the projective limit as well admitting a canonical basis.
The role of the canonical basis for the idempotented quantum group $\dot{\mbf U}_q(\mak{gl}_n)$ is similar to that of Kazhdan-Lusztig
bases \cite{KL79} for Iwahori-Hecke algebras.
Subsequently, the Schur-Jimbo duality, as a bridge
connecting the Iwahori-Hecke algebra of type $\mrm A$ and (idempotented) quantum group $\mbf U_q(\mak{gl}_n)$, is realized
geometrically by considering the product variety of the complete flag variety and the $n$-step
partial flag variety of type $\mrm A$ in \cite{GL92}.

There have been some generalizations of the BLM-type construction using the $n$-step (partial) flag varieties of affine type $\mrm A$ earlier on; see Ginzburg-Vasserot \cite{GV93} and Lusztig \cite{Lu99,Lu00}.
Some further developments could be found in \cite{GRV93,Mc12,P09,VV99}.
Also, there is an affine version of Schur-Jimbo duality formed in \cite{CP96}.
The affine quantum Schur algebra $\mbf S_{n,d}$ is by definition the convolution algebra of pairs of flags of affine type $\mrm A$.
There is a natural homomorphism from the idempotented quantum affine $\mak{sl}_n$ to the affine quantum Schur algebra
\begin{equation*}
\dot{\mbf U}_q(\wih{\mak{sl}}_n) \longrightarrow \mbf S_{n,d},
\end{equation*}
which is no longer surjective.
The image of this map is denoted by Lusztig subalgebra $\mbf U_{n,d}$ and generated by the Chevalley generators.

An $(\mbf U_q^\imath(\mak{gl}_n), \mbf H_\mrm C)$-duality was discovered algebraically and categorically in \cite{BW13} as a crucial ingredient for a new approach to Kazhdan-Lusztig theory of classical types.
Motivated by \cite{BW13}, Bao, Kujawa, Li and Wang \cite{BKLW14, BLW14} provided a geometric construction of Schur-type algebras $\mbf S_{n,d}^{\imath,\rm fin}$ in terms of $n$-step flag varieties of type $\mrm B_d$ (or $\mrm C_d$).
The $\imath$-Schur duality has been realized in \cite{BKLW14} by using mixed pairs of $n$-step flags and complete flags of type $\mrm B$ or $\mrm C$.
They further established multiplication formulas in the Schur algebras $\mbf S_{n,d}^{\imath, \rm fin}$ with divided powers of Chevalley generators, which again enjoy some remarkable stabilization properties as $d \rightarrow \infty$.
They showed the quantum algebra arising from the stabilization procedure is a coideal subalgebra $\mbf U_q^\imath (\mak{gl}_n)$ of $\mbf U_q(\mak{gl}_n)$.
The pair $(\mbf U_q^\imath (\mak{gl}_n),\mbf U_q(\mak{gl}_n))$ forms a quantum symmetric pair.

The theory of quantum symmetric pairs was systematically developed by Letzter \cite{Le99, Le02} and Kolb \cite{Ko14}.
Let $\imath$ be an algebra involution on the symmetrizable Kac-Moody algebra $\mak g$ of the second kind \cite[\S 2]{Ko14}, and $\mak g^\imath$ be the subalgebra of $\imath$-invariants in $\mak g$.
For simple Lie algebras $\mak g$ of finite type, the classification of $\mak g^\imath$ corresponds to the classification of real simple Lie algebras \cite{OV90}.
The quantum analogue $\mbf U_q^\imath(\mak g)$ of the enveloping algebra $\mbf U(\mak g^\imath)$ is a coideal subalgebra of the quantized enveloping algebra $\mbf U_q(\mak g)$ in the sense that the comultiplication $\Delta$ on $\mbf U_q(\mak g)$ satisfies
\begin{equation*}
\Delta: \mbf U_q^\imath (\mak g) \longrightarrow \mbf U_q^\imath (\mak g)\otimes \mbf U_q(\mak g).
\end{equation*}
The algebra $\mbf U_q^\imath(\mak g)$ specializes to $\mbf U(\mak g^\imath)$ at $q=1$ and the pair $(\mbf U_q^\imath(\mak g), \mbf U_q(\mak g))$ is called a quantum symmetric pair.

The affine Schur algebra $\mbf S_{n,d}^{\mak c}$ is by the definition the convolution algebra of pairs of flags of affine type $\mrm C$ \cite{FLLLW20}.
It admits a canonical basis which enjoys a positivity with respect to multiplication. Denote a subalgebra $\mbf U_{n,d}^{\mak c}$ of $\mbf S_{n,d}^{\mak c}$ generated by the Chevalley generators.
They introduce a comultiplication homomorphism and transfer map on Schur algebras such that comultiplication homomorphism \cite[\S 5]{FLLLW20} and transfer map \cite[\S 6]{FLLLW20} make sense on the level of Schur algebras instead of Lusztig algebras. The algebra $\mbf U_n^{\mak c}$ is by definition a suitable subalgebras of the projective limit of the projective system of Lusztig algebras and the comultiplication homomorphisms gives rise to show that $\mbf U_n^{\mak c}$ is a coideal subalgebra of $\mbf U_q(\widehat{\mak{sl}}_n)$.
The idempotented form of $\mbf U_n^{\mak c}$, denoted by $\dot{\mbf U}_n^{\mak c}$ can be formulated analogous to the idempotented quantum groups $\dot{\mbf U}_n$ in \cite{BLM90,Lu93}.
The canonical basis of $\dot{\mbf U}_n^{\mak c}$ be constructed and its positivity with respect to the multiplication and a bilinear pairing of geometric origin also be established in \cite{FLLLW20}.

Fan, Li \cite{FL14} established a new duality between the Schur algebra $\mbf S^m$ and the Iwahori-Hecke algebra $\mbf H_\mrm D$ of type $\mrm D$ attached to the special orthogonal group ${\rm SO}_\mbb F(2d)$ algebraically and geometrically.
The algebra $\mbf S^m$ contains the subalgebra $\wit{\mbf S}^m$ and two additional idempotents, where $\wit{\mbf S}^m$ isomorphic to $\mbf S_{n,d}^{\imath, \rm fin}$ of type $\mrm C$.
They provided a geometric construction of the Schur algebra $\mbf S^m$ in terms of $n$-step flag varieties of type $\mrm D$.
The Schur duality has been realized by using mixed pairs of $n$-step flags and complete flags of type $\mrm D$.
The quantum algebra $\mbf U^m$ and its idempotented form (canonical basis) are obtained by stabilization and completion process following \cite{BLM90,BKLW14}.

To this end, it is natural to ask what happens to the case of the affine type $\mrm D$.
The purpose of this paper is to provide an answer to this question, as a sequel to \cite{FL14, FLLLW20}.
There is a lattice presentation of the complete and $n$-step flag variety of affine type $\mrm D$ in \cite{CFW24},
on which the special orthogonal group ${\rm SO}_F(V)$ (where $F=k((\var))$) acts on the complete flags and the $n$-step partial flag variety $\mcal X_{n,d}^\mak d$, which is formulated in this paper, for $n$ even.

We parameterize the orbits for the product $\mcal X_{n,d}^\mak d\times \mcal X_{n,d}^\mak d$ under the diagonal action of the group ${\rm SO}_F(V)$ by the set $\Xi_\mak d$ of signed matrices.
Similarly, denote by $\Xi_\mak d^{ap}$ the subset of aperiodic signed matrices in $\Xi_\mak d$.
The Schur algebra $\mbf S_{n,d}^\mak d$
is by the definition the convolution algebra of pairs of the flags in $\mcal X_{n,d}^\mak d$.
We further establish the Schur duality between the (extended) Hecke algebra $\mbf H_{
\wit{\mrm D}}$ and the Schur algebra $\mbf S_{n,d}^\mak d$ algebraically and geometrically.
 The algebra $\mbf S_{n,d}^\mak d$ admits a standard basis $\{[\mak a] \mid \mak a \in \Xi_\mak d \}$ and a canonical basis $\{\{\mak a\}_d \mid \mak a \in \Xi_\mak d \}$, which enjoys a positivity with respect to multiplication.
We formulate a subalgebra $\mbf U_{n,d}^\mak d$ of $\mbf S_{n,d}^\mak d$ generated by the Chevalley generators.
Our first main result is the following:
\begin{thrm}[Theorem \ref{1222}]
The algebra $\mbf U_{n,d}^\mak d$ admits a monomial basis $\{\zeta_\mak a \mid \mak a \in \Xi_\mak d^{ap}\}$.
Similarly, it possesses a canonical basis $\{\{\mak a\}_d \mid \mak a \in \Xi_\mak d^{ap}\}$, which is compatible with the corresponding bases in $\mbf S_{n,d}^\mak d$ under the inclusion $\mbf U_{n,d}^\mak d \subset \mbf S_{n,d}^\mak d$.
\end{thrm}

Generalizing the constructions in affine type $\mrm A, \mrm C$ and finite type $\mrm D$ \cite{FL15,FLLLW20,Lu00}, we introduce a subalgebra $\wit{\mbf S}_{n,d}^{\mak d}$ of $\mbf S_{n,d}^{\mak d}$ and the comultiplication-like homomorphisms $\wit{\Delta}^{\mak d}=\wit{\Delta}_{d',d''}^\mak d:
\wit{\mbf S}_{n,d}^{\mak d}\rightarrow
\wit{\mbf S}_{n,d}^{\mak d}\otimes \mbf S_{n,d''}$ for any composition $ d=d'+d''$.
These further lead to the transfer maps $\phi_{d,d-n}^\mak d: \mbf S_{n,d}^\mak d\rightarrow \mbf S_{n,d-n}^\mak d$, which are shown to preserve the Chevalley generators.
The homomorphisms $\phi_{d,d-n}^\mak d$ make sense on the level of Schur algebras instead of Lusztig algebras.

The algebras $\wit{\mbf U}_n^\mak d$ and $\mbf U_n^\mak d$ are by definition the suitable subalgebras of projective limit of the projective system $\{(\mbf U_{n,d}^\mak d,\phi_{d,d-n}^\mak d)\}_{d\geq 1}$, just as $\mbf U_n$ is a limit algebra for a similar affine type $\mrm A$ projective system.
We showed that the family of homomorphisms $\{\wit{\Delta}_{d',d''}^\mak d\}$ gives rise to a homomorphism $\wit{\Delta}^\mak d:\wit{\mbf U}_n^\mak d \rightarrow \wit{\mbf U}_n^\mak d\otimes \mbf U_n$ and an injective homomorphism $\jmath_n:\wit{\mbf U}_n^\mak d \rightarrow \mbf U_n$, whose images on the Chevalley generators are explicitly given.
\begin{thrm}[Proposition \ref{1400}, Proposition \ref{1900}]
There exists a decomposition of $\mbf U_n^\mak d$ such that $\mbf U_n^\mak d=\mbf J_+\wit{\mbf U}_n^\mak d
\oplus\mbf J_-\wit{\mbf U}_n^\mak d$, where $\mbf J_+,\mbf J_-$ are two idempotents, and the algebra $\wit{\mbf U}_n^\mak d$ is a coideal subalgebra of $\mbf U_n$.
The pair $(\mbf U_n, \wit{\mbf U}_n^\mak d)$ forms a quantum symmetric pair of affine type in the sense of Lezter and Kolb \cite{Ko14}.
\end{thrm}

We formulate idempotented forms of $\mbf U_n^\mak d$, denoted by $\dot{\mbf U}_n^\mak d$, which is analogous to the idempotented quantum groups as formulated in \cite{FL14,FLLLW20}.
Following the approach of \cite{Mc12} in the affine type $\mrm A$ setting and \cite{FL15} in the finite type $\mrm D$ setting, we construct canonical basis for $\dot{\mbf U}_n^\mak d$ and establish its positivity with respect to the multiplication and a bilinear pairing of geometric origin.
\begin{thrm}[Theorem \ref{1500},Theorem \ref{1501}]
The algebra $\dot{\mbf U}_n^\mak d$ admits a canonical basis $\dot{\mbf B}_n^\mak d$. The structure constants of the canonical basis $\dot{\mbf B}_n^\mak d$ lie in $\mbb N[v,v^{-1}]$ with respect to the multiplication and in $v^{-1}\mbb N[[v]]$ with respect to the bilinear pairing.
\end{thrm}
\noindent {\bf Acknowledgement.}
Z. Fan was partially supported by the NSF of China grant
12271120, the NSF of Heilongjiang Province grant JQ2020A001, and the Fundamental
Research Funds for the central universities.

\section{Convolution algebra and Schur Duality}
In this section, we study the convolution algebra $\mbf S_{n,d}^\mak d$ of pairs of $n$-step flags of affine type $\mrm D$.
Then we present multiplication formulas in $\mbf S_{n,d}^\mak d$ with (the divided powers of) Chevalley generators.
\subsection{Lattice presentation of affine flag varieties}
In this subsection, we recall some results about affine flag varieties of type $\mrm D$.
We refer to \cite{CFW24} and the references therein for more details.

Let $\mbb F$ be a finite field of $q$ elements with odd characteristic.
Let $F=\mbb F((\var))$ be the field of formal Laurent series over $\mbb F$ and
$\mak o=\mbb F[[\var]]$ the ring of formal power series.
Denote by $\mak m$ the maximal ideal of $\mak o$ generated by
$\var$.
Throughout this paper, we shall fix a pair $(d,r)$ of positive integers and
\begin{equation*}
D=2d, \ \  n=2r.
\end{equation*}
Let $V=F^D$ be a vector space with a symmetric bilinear form
$Q: V\times V \rightarrow  F$ whose associated matrix under the standard basis is
\begin{equation}\label{matrix}
J=\begin{bmatrix}
0 & I_d \\
I_d & 0 \\
\end{bmatrix}.
\end{equation}
Denote by ${\rm O}_F(V)$ and ${\rm SO}_F(V)$ the orthogonal group and the special orthogonal group with respect to $Q$, respectively.

A free $\mak o$-submodule $\mcal L$ of $V$ of rank $D$ is called an $\mak o$-lattice.
We set
\begin{equation*}
  \mcal L^\sharp =\{v\in V \mid Q(v,\mcal L) \subset\mak o\}, \ \
  \mcal L^\ast =\{v\in V \mid Q(v,\mcal L ) \subset\mak m\}.
\end{equation*}

Let $\mcal M$ and $ \mcal L$ be two lattices such that $\mcal M \subset \mcal L$.
Denote by $|\mcal L/\mcal M|$ the dimension of the $ \mbb F$-vector space $\mcal L / \mcal M$.

Fix a pair $(\mcal M, \mcal L)$ of lattices such that $\var \mcal L \subset \mcal M=\mcal M^\ast\subset \mcal L=\mcal L^\sharp$.
We consider the sets $\mcal Y_d^\mak d$ and $\mcal X_{n,d}^\mak d$ of all collections of lattices in $V$:
\begin{align*}
\mcal X_{n,d}^\mak d=&\{L=(L_i)_{i\in \mbb Z}\mid L_i \subset L_{i+1}, L_i=\var L_{i+n}, L_i^\ast=L_{n-i}, \ \forall  i \in \mbb Z \};\\
\mcal Y_d^\mak d=&\{\Lambda=(\Lambda_i)_{i\in \mbb Z} \mid
\Lambda_i\subset \Lambda_{i+1}, \Lambda_i=\var \Lambda_{i+D}, \ \Lambda_i^\ast=\Lambda_{D-i},  \ |\Lambda_i/\Lambda_{i-1}|=1, \\ \ & \forall   i \in \mbb Z, \
|\mcal M/\mcal M\cap \Lambda_d|+|\mcal L/\mcal L\cap \Lambda_D|={\rm even}\}.
\end{align*}
Then ${\rm SO}_F(V)$ acts on $\mcal X_{n,d}^\mak d$ and $\mcal Y_d^\mak d$ by component-wise action.
Denote by $\mcal X_{n,d}^{\mak d,1}$ and $\mcal X_{n,d}^{\mak d,2}$ the subsets of $\mcal X_{n,d}^\mak d$ as following:
 \begin{align*}
  \mcal X_{n,d}^{\mak d,1}&=\left\{L=(L_i)_{i\in \mbb Z} \in \mcal X_{n,d}^\mak d  \mid  |\mcal M/L_r\cap \mcal M|+|\mcal L/L_n\cap \mcal L|={\rm even}\right\}; \\
   \mcal X_{n,d}^{\mak d,2}&=\left\{L=(L_i)_{i\in \mbb Z} \in \mcal X_{n,d}^\mak d  \mid |\mcal M/L_r\cap \mcal M|+|\mcal L/L_n\cap \mcal L|={\rm odd}\right\}.
 \end{align*}
 Then $\mcal X_{n,d}^\mak d$ can be decomposed into $\mcal X_{n,d}^\mak d=\mcal X_{n,d}^{\mak d,1}\sqcup \mcal X_{n,d}^{\mak d,2}$.
 We set
\begin{align}\label{abcd}
 \Lambda_{n,d}^\mak d=\bigg\{\lambda=(\lambda_i)_{i\in \mbb Z}\in \mbb N^\mbb Z \ \Big| \ \lambda_i=\lambda_{1-i}=\lambda_{i+n}, \forall i\in \mbb Z; \sum_{1\leq i\leq n}\lambda_i=D\bigg\}.
\end{align}
 The set $\mcal X_{n,d}^\mak d$ admits the following decomposition.
 \begin{align*}
\mcal X_{n,d}^\mak d=\bigsqcup_{\mbf a=(a_i)\in \Lambda_{n,d}^\mak d}\mcal X_{n,d}^\mak d(\mbf a),\ \  \text{where} \ \mcal X_{n,d}^\mak d(\mbf a)=\left\{L\in \mcal X_{n,d}^\mak d \mid |L_i/L_{i-1}|=a_i, \forall i \in \mbb Z\right\}.
\end{align*}

Let ${\rm SO}_F(V)$ act diagonally on $\mcal X_{n,d}^\mak d \times \mcal X_{n,d}^\mak d, \mcal X_{n,d}^\mak d \times \mcal Y_d^\mak d$ and $\mcal Y_d^\mak d \times \mcal Y_d^\mak d$, respectively.
Consider the following sets of matrices with entries being non-negative integer as following:
\begin{align*}
\Xi_{n,d}=&\bigg\{A \in  {\rm Mat}_{\mbb Z \times \mbb Z }(\mbb N) \ \Big| \ a_{i,j}=a_{1-i,1-j}=a_{i+n,j+n}, \forall i,j \in \mbb Z, \sum_{1\leq i\leq n,j}a_{ij}=D \bigg\},\\
\Pi_{n,d}=&\bigg\{B \in  {\rm Mat}_{\mbb Z \times\mbb Z}(\mbb N) \ \Big| \ b_{i,j}=b_{1-i,1-j}=b_{i+n,j+D}, \forall i,j \in \mbb Z, \sum_jb_{ij}=1 \bigg\},\\
\Sigma_d=&\bigg\{\sigma \in  {\rm Mat}_{\mbb Z \times \mbb Z}(\mbb N)|\sigma_{i,j}=\sigma_{1-i,1-j}=
\sigma_{i+D,j+D}, \forall i,j \in \mbb{Z}, \sum_{i}a_{ij}=\sum_{j}a_{ij}=1,\\ &\sum_{i\leq 0<j}a_{ij}+\sum_{i\leq d<j}a_{ij}={\rm even} \bigg\}.\\
\end{align*}

We define a map $\widetilde{\Phi}$ from the set of ${\rm SO}_{F}(V)$-orbits in $\mcal X_{n,d}^\mak d\times \mcal X_{n,d}^\mak d$ to $\Xi_{n,d}$ by sending the orbit ${\rm SO}_F(V)\cdot(L,L')$ to $A=(a_{ij})_{i,j\in \mbb Z}$, where
\begin{equation}
a_{ij}=\bigg|\frac{L_{i-1}+L_i\cap L_j'}{L_{i-1}+L_i\cap L_{j-1}'}\bigg|.
\end{equation}

For any $g \in {\rm O}_F(V)\backslash {\rm SO}_F(V)$,
there exists a bijection between $\mcal X_{n,d}^{\mak d,1}$ and $\mcal X_{n,d}^{\mak d,2}$ defined by $L\mapsto g\cdot L$,
which yields the following bijections.
\begin{equation}\label{sign}
\begin{split}
{\rm SO}_F(V)\backslash \mcal X_{n,d}^{\mak d,1}\times \mcal X_{n,d}^{\mak d,2}\rightarrow {\rm SO}_F(V)\backslash \mcal X_{n,d}^{\mak d,2}\times \mcal X_{n,d}^{\mak d,1};\\
{\rm SO}_F(V)\backslash \mcal X_{n,d}^{\mak d,1}\times \mcal X_{n,d}^{\mak d,1}\rightarrow {\rm SO}_F(V)\backslash \mcal X_{n,d}^{\mak d,2}\times \mcal X_{n,d}^{\mak d,2}.
\end{split}
\end{equation}
Moreover, corresponding pairs on both sides under the bijections in \eqref{sign} get sent to the same matrix by $\widetilde{\Phi}$.
In corresponding to \eqref{sign}, we define a sign function
\begin{equation}
{\rm sgn}(i,j)=
\begin{cases}
+  &  \mbox{if}\ (i,j)=(1,1),(1,2);\\
-  &  \mbox{if}\ (i,j)=(2,1),(2,2).
\end{cases}
\end{equation}
We set
\begin{equation*}
\Xi_{n,d}^+=\Xi_{n,d}\times \{+\}; \ \ \Xi_{n,d}^-=\Xi_{n,d}\times \{-\}.
\end{equation*}
Moreover, we set that
\begin{equation*}
\Xi_\mak d=\Xi_{n,d}^+\sqcup \Xi_{n,d}^-.
\end{equation*}
Elements in $\Xi_\mak d$ will be called signed matrices.
To each signed matrix $\mak a=(A, \alpha)\in \Xi_{\mak d}$, we define its row and columu sum vectors ${\rm ro}(\mak a)=({\rm ro}(\mak a)_i)_{i\in \mbb Z}$ and ${\rm co}(\mak a)=({\rm co}(\mak a)_i)_{i\in \mbb Z}$ in $\Lambda_{n,d}^\mak d$ by
\begin{align*}
{\rm ro}(\mak a)_i=\sum_{j\in \mbb Z}a_{ij}, \ \  {\rm co}(\mak a)_j=\sum_{i\in \mbb Z}a_{ij},  \ \  \forall i,j \in \mbb Z.
\end{align*}

Similarly, we define maps $\Phi'$ and $\Phi''$ from the sets of ${\rm SO}_F(V)$-orbits in $\mcal X_{n,d}^\mak d \times \mcal Y_d^\mak d$
and $\mcal Y_d^\mak d\times \mcal Y_d^\mak d$ to $\Pi_{n,d}$ and $\Sigma_d$,
respectively, by sending the orbit ${\rm SO}_F(V)\cdot(L,\Lambda)$ to $B=(b_{ij})_{i,j\in \mbb Z}$
and ${\rm SO}_F(V)\cdot(\Lambda,\Lambda')$ to $\sigma=(\sigma_{ij})_{i,j\in \mbb Z}$, where
\begin{equation*}
b_{ij}=\bigg|\frac{L_i+L_i\cap \Lambda_j}{L_{i-1}+L_i\cap \Lambda_{j-1}}\bigg|, \ \ \sigma_{ij}=\bigg|\frac{\Lambda_i+\Lambda_i\cap \Lambda_j'}{\Lambda_{i-1}+\Lambda_i\cap \Lambda_{j-1}'}\bigg|.
\end{equation*}

By a similar argument as for \cite[Proposition 3.2.2]{FLLLW20}, we have the following proposition.

\begin{proposition}\label{isomorphism}
 The map $\wit{\Phi}$ induces a bijection
\begin{equation*}
\Phi:  {\rm SO}_F(V)\backslash\mcal X_{n,d}^\mak d\times \mcal X_{n,d}^\mak d \rightarrow \Xi_\mak d,  \ \ \ {\rm SO}_F(V)\cdot(L,L')\mapsto (A, \alpha),
\end{equation*}
where $A=\wit{\Phi}({\rm SO}_F(V) \cdot (L,L'))$ and $\alpha={\rm sgn}(i,j)$, if $(L,L') \in \mcal X_{n,d}^{\mak d,i}\times \mcal X_{n,d}^{\mak d,j}$.
\end{proposition}

For any $\mak a=(A, \alpha) \in \Xi_\mak d$, we denote by $\mcal{O}_\mak a$ the orbit corresponding to a signed matrix $\mak a$, and introduce the following notation.
\begin{equation}
{\rm sup}(\mak a)=(i,j),  \ \  {\rm if} \  \mcal O_\mak a \subset \mcal X_{n,d}^{\mak d,i}\times \mcal X_{n,d}^{\mak d,j}.
\end{equation}

The maps $\Phi'$ and $\Phi''$ are bijective from Proposition \ref{isomorphism} and \cite[Proposition 4.2]{CFW24}.

\subsection{The double centralizer}
Let $v$ be an indeterminate and $\mcal A=\mbb Z[v,v^{-1}]$.
We define
\begin{align*}
\mbf S_{n,d;\mcal A}^\mak d=\mcal A_{{\rm SO}_F(V)}(\mcal X_{n,d}^\mak d\times \mcal X_{n,d}^\mak d)
\end{align*}
to be the space of ${\rm SO}_F(V)$-invariant $\mcal A$-valued (generic) functions
on $\mcal X_{n,d}^\mak d\times \mcal X_{n,d}^\mak d$.
For $\mak a\in \Xi_\mak d$,
we denote by $e_\mak a$ the characteristic function of the corresponding orbit $\mcal O_\mak a$.
Then $\mbf S_{n,d;\mcal A}^\mak d$ is a free $\mcal A$-module with basis $\{e_\mak a \mid \mak a\in \Xi_\mak d \}$.
Let $\mak a, \mak b, \mak c \in \Xi_\mak d$.
We fix $L, L' \in \mcal X_{n,d}^\mak d$ such that $(L, L') \in \mcal O_\mak c$ and set
\begin{equation*}
g_{\mak a, \mak b, \mak c; q}=\sharp \big\{\widetilde{L}\in \mcal X_{n,d}^ \mak d \mid (L, \widetilde{L}) \in \mcal O_\mak a, (\widetilde{L}, L') \in \mcal O_\mak b \big\}.
\end{equation*}
It is well known that there exists a polynomial $g_{\mak a, \mak b, \mak c} \in \mbb Z[v,v^{-1}]$ such that $g_{\mak a, \mak b, \mak c ;q}=g_{\mak a, \mak b, \mak c}|_{v=\sqrt{q}}$.
We define a convolution product $\ast$ as follows:
\begin{equation*}
e_\mak a\ast e_\mak b=\sum_\mak c g_{\mak a, \mak b, \mak c}e_\mak c.
\end{equation*}
Then we set
\begin{equation*}
\mbf S_{n,d}^\mak d=\mbb Q(v)\otimes_\mcal A\mbf S _{n,d;\mcal A }^\mak d.
\end{equation*}
The algebras $\mbf S_{n,d;\mcal A}^\mak d$ and $\mbf S_{n,d}^\mak d$ are called the Schur algebras of affine type $\mrm D$.

Similarly, we define the free $\mcal A$-modules
\begin{align*}
\mbf V_{n,d;\mcal A}^\mak d=\mcal A_{{\rm SO}_F(V)}(\mcal X_{n,d}^\mak d\times \mcal Y_d^\mak d) \ \ \  {\rm and} \ \ \ \mbf H_{d;\mcal A}^\mak d=\mcal A_{{\rm SO}_F(V)}(\mcal Y_d^\mak d\times \mcal Y_d^\mak d).
\end{align*}
A similar convolution product gives an associative algebra structure on $\mbf H_{d;\mcal A}^\mak d$.
According to \cite[Proposition 4.6]{CFW24}, we have $\mbf H_{d;\mcal A}^\mak d$ is isomorphic to the (extended) affine Hecke algebra $\mbf H_{\wit {\mrm D}}$.

There is a left $\mbf S_{n,d;\mcal A}^\mak d$-action and a right $\mbf H_{d;\mcal A}^\mak d$-action on $\mbf V_{n,d;\mcal A}^\mak d$.
For any $f \in \mbf S_{n,d; \mcal A}, g \in \mbf V_{n,d;\mcal A}^\mak d$ and $h \in \mbf H_{d; \mcal A}^\mak d$,
\begin{equation*}
f\cdot g (L,\Lambda)=\sum_{L' \in \mcal X_{n,d}^\mak d}f(L,L')g(L',\Lambda),  \ \
g\cdot h (L,\Lambda)=\sum_{\Lambda' \in \mcal Y_d^\mak d}g(L,\Lambda')h(\Lambda',\Lambda).
\end{equation*}
Moreover, these two actions commute and hence we
have the following $\mcal A$-algebra homomorphisms:
\begin{equation*}
\mbf S_{n,d;\mcal A}^\mak d\rightarrow {\rm End}_{\mbf H_{d;\mcal A}^\mak d}(\mbf V_{n,d;\mcal A}^\mak d) \ \ \ {\and} \ \ \ \mbf H_{d;\mcal A}^\mak d \rightarrow {\rm End}_{\mbf S_{n,d;\mcal A}^\mak d}(\mbf V_{n,d;\mcal A}^\mak d).
\end{equation*}
By \cite[Theorem 2.1]{P09}, we have the following double centralizer property.
\begin{lemma}
${\rm End}_{\mbf H_{d;\mcal A}^\mak d}(\mbf V_{n,d;\mcal A}^\mak d)\simeq \mbf S_{n,d;\mcal A}^\mak d$ and ${\rm End}_{\mbf S_{n,d;\mcal A}^\mak d}(\mbf V_{n,d;\mcal A}^\mak d) \simeq \mbf H_{d;\mcal A}^\mak d$, if $n\leq D$.
\end{lemma}
\subsection{Canonical bases of $\mbf S_{n,d;\mcal A}^\mak d$}
In this subsection, we assume that the ground field is an algebraic field is an algebraic closure $\overline{\mbb F}$ of $\mbb F$ when we talk about the dimension of a ${\rm SO}_F(V)$-orbit or its stabilizer.

Fix $L \in \mcal X_{n,d}^\mak d$.
For $\mak a \in \Xi_\mak d$, we set
\begin{align*}
X_\mak a^L=\left\{ L' \in \mcal X_{n,d}^\mak d \mid (L,L') \in \mcal O_\mak a \right\}.
\end{align*}
This is an orbit of the stabilizer subgroup ${\rm Stab}_{{\rm SO}_F(V)}(L)$ of ${\rm SO}_F(V)$, and one can associate to it a structure of quasi-projective algebraic variety. Now, we compute its dimension $d(\mak a)$.

We have the following analogue of \cite[lemma 4.1.1]{FLLLW20}
\begin{lemma}\label{matric}
 Fix $L \in \mcal X_{n,d}^\mak d$.
 For $\mak a \in \Xi_\mak d$,
 the dimension of $X_\mak a^L$ is given by
\begin{equation}\label{Dimen eq}
  d(\mak a)=\frac{1}{2}\bigg(\sum_{\substack{i\geq k, j<l\\ i\in [1,n]}}a_{ij}a_{kl}-\sum_{i\geq 1>j}a_{ij}-\sum_{i\geq r+1>j}a_{ij}\bigg).
\end{equation}
\end{lemma}

 For any $\mak a \in \Xi_\mak d$, let
 \begin{equation*}
 [\mak a]=v^{-d(\mak a)}e_\mak a.
 \end{equation*}

 Denote by $A^t$ the transposition matrix of $A$. For any $\mak a=(A,\alpha) \in \Xi_\mak d$, we define $\mak a^t=(A^t,\alpha')$, where
\begin{align*}
\alpha'=
\begin{cases}
\alpha  &  \mbox{if \ sup}\ (\mak a)=(1,1),(2,2);\\
-\alpha  &  \mbox{otherwise}.
\end{cases}
\end{align*}
\begin{remark}
For any $\mak a \in \Xi_\mak d$, we have
\begin{align*}
d(\mak a)-d(\mak a^t)=\frac{1}{4}\sum_{i\in [1,n]}a_{ij}^2-\frac{1}{4}\sum_{j\in [1,n]}a_{ji}^2.
\end{align*}
Hence the $\mbb Q(v)$-linear involution $\Psi:\mbf S_{n,d}^\mak d\rightarrow \mbf S_{n,d}^\mak d$ defined by $[\mak a]\mapsto [\mak a^t]$ is an algebra anti-automorphism.
\end{remark}

Define a partial order $"\leq"$ on $\Xi_\mak d$ by $\mak a \leq \mak b$ if $\mcal O_\mak a \subset \overline{\mcal O_\mak b}$.
For any $\mak a=(A, \alpha), \mak b=(B, \alpha')$ in $\Xi_\mak d$,
we say that $\mak a\preceq \mak b$ if and only if $\alpha=\alpha'$ and the following condition hold.
\begin{align*}
\sum_{k\geq i, l\leq j }a_{k,l} \leq \sum_{k\geq i, l\leq j}b_{k,l},  \ \ \forall i>j.
\end{align*}
Since the Bruhat order of affine type $\mrm D$ is compatible with the Bruhat order of affine type $\mrm A$,
we see that the partial order "$\leq$" is compatible with (though possibly weaker than) the Bruhat order of affine type $\mrm D$.

Let $IC_\mak a$ be the intersection cohomology complex of the closure $\overline{X_\mak a^L}$ of $X_\mak a^L$,
taken in certain ambient algebraic variety over $\overline{\mbb F}$, such that the restriction of the stratum $IC_\mak a$ to $X_\mak a^L$ is the constant sheaf on $X_\mak a^L$.
We refer to \cite{BBD82} for the precise definition of intersection complexes.
The restriction of the $i$-th cohomology sheaf $\mcal H_{X_\mak b^L}^i(IC_\mak a)$ of $IC_\mak a$ to $X_\mak b^L$ for $\mak b \leq \mak a$ is a trivial local system, whose rank is denoted by $n_{\mak b,\mak a,i}$.
 We set
\begin{align*}
\{\mak a\}_d=\sum_{\mak b\leq \mak a}P_{\mak b,\mak a}[\mak b], \ \  {\rm where} \ P_{\mak b,\mak a}=\sum_{i\in \mbb Z}n_{\mak b,\mak a,i}v^{i-d_\mak a+d_\mak b}.
\end{align*}
The polynomials $P_{\mak b,\mak a}$ satisfy
\begin{align*}
P_{\mak a,\mak a}=1, \ \ {\rm and}\ P_{\mak b,\mak a} \in v^{-1}\mbb Z[v^{-1}] \ {\rm for \ any}\ \mak b\leq \mak a.
\end{align*}
Recall that the set $\{[\mak a] \mid \mak a \in \Xi_\mak d\}$ forms a $\mbb Q(v)$-basis of $\mbf S_{n,d}^\mak d$.
We have the following proposition.

\begin{proposition}
The set $\left\{\{\mak a\}_d \mid \mak a \in \Xi_\mak d \right\}$ forms an $\mcal A$-basis of $\mbf S_{n,d,\mcal A}^\mak d$ and a $\mbb Q(v)$-basis of $\mbf S_{n,d}^\mak d$, called the canonical basis. Moreover, the structure constants of $\mbf S_{n,d}^\mak d$ with respect to the canonical basis are in $\mbb N[v,v^{-1}]$.
\end{proposition}

\subsection{Multiplication formulas of Chevalley generators}
Given $i,j \in \mbb Z$, let $E^{ij}$ be the $\mbb Z\times \mbb Z$ matrix whose $(k,l)$-th entries are $1$, for all $(k,l)\equiv(i,j) \ {\rm mod} \ n$, and $0$ otherwise.
We set
\begin{equation*}
E_\theta^{ij}=E^{ij}+E^{1-i,1-j}.
\end{equation*}

We now give the multiplication formulas of Chevalley  generators.
The following lemma is a counterpart of \cite[Proposition 3.5]{Lu99}.
\begin{lemma} \label{formulas}
Let $\mak a=(A, \alpha) \in \Xi_\mak d$.
Assume that $h \in [1,r-1]$.

$(a)$ Assume that $\mak b=(B, \alpha) \in \Xi_\mak d$ such that $B-E_\theta^{h,h+1}$ is a diagonal matrix and ${\rm co}(\mak b)={\rm ro}(\mak a)$.
Then
\begin{equation}
e_\mak b\ast e_\mak a=\sum_{p\in \mbb Z,a_{h+1,p}\geq 1}v^{2\sum_{j>p}a_{h,j}}\frac{v^{2(1+a_{h,p})}-1}{v^2-1}e_{(A+E_\theta^{h,p}-
E_\theta^{h+1,p}, \alpha)}.
\end{equation}

$(b)$ Assume that $\mak c=(C, \alpha) \in \Xi_\mak d$ such that $C-E_\theta^{h+1,h}$ is a diagonal matrix and ${\rm co}(\mak c)={\rm ro}(\mak a)$.
Then
\begin{equation}
e_\mak c\ast e_\mak a=\sum_{p\in \mbb Z,a_{h,p}\geq 1}v^{2\sum_{j<p}a_{h+1,j}}\frac{v^{2(1+a_{h+1,p})}-1}{v^2-1}e_{(A+E_\theta^{h+1,p}-
E_\theta^{h,p}, \alpha)}.
\end{equation}

$(c)$ Assume that $\mak n=(D, -\alpha) \in \Xi_\mak d$  such that $D-E_\theta^{0,1}$ is a diagonal matrix and ${\rm co}(\mak n)={\rm ro}(\mak a)$.
Then
\begin{equation}
\begin{split}
e_\mak n\ast e_\mak a&=\sum_{p\leq 0,a_{0,p}\geq 1}v^{2\sum_{j<p}a_{1,j}}\frac{v^{2(1+a_{1,p})}-1}{v^2-1}e_{(A+E_\theta^{1,p}-
E_\theta^{0,p}, -\alpha)}\\
&+\sum_{p> 0 ,a_{0,p}\geq 1}v^{2(\sum_{j<p}a_{1,j}-1)}\frac{v^{2(1+a_{1,p})}-1}{v^2-1}e_{(A+E_\theta^{1,p}-
E_\theta^{0,p}, -\alpha)}.
\end{split}
\end{equation}

$(c')$ Assume that $\mak n'=(D', -\alpha) \in \Xi_\mak d$  such that $D-E_\theta^{r,r+1}$ is a diagonal matrix and ${\rm co}(\mak n')={\rm ro}(\mak a)$.
Then
\begin{equation}
\begin{split}
e_{\mak n'}\ast e_\mak a&=\sum_{p\leq r,a_{r,p}\geq 1}v^{2\sum_{j<p}a_{r+1,j}}\frac{v^{2(1+a_{r+1,p})}-1}{v^2-1}e_{(A+E_\theta^{r+1,p}-
E_\theta^{r,p}, -\alpha)}\\
&+\sum_{p> r,a_{r,p}\geq 1}v^{2(\sum_{j<p}a_{r+1,j}-1)}\frac{v^{2(1+a_{r+1,p})}-1}{v^2-1}e_{(A+E_\theta^{r+1,p}-
E_\theta^{r,p}, -\alpha)}.
\end{split}
\end{equation}
\end{lemma}
\begin{proof}
The proof of cases $(a)$ and $(b)$ is
essentially the same as the proof of \cite[Proposition 3.5]{Lu99}.
We now prove the case $(c)$ as follows. Without loss of generality, we assume that $\alpha=+$.
Let $L=(L_i)_{i\in \mbb Z}, L'=(L_i')_{i \in \mbb Z}$ and $L''=(L_i'')_{i\in \mbb Z}$ be filtrations of lattices in $\mcal X_{n,d}^\mak d$
such that $(L,L')\in \mcal O_\mak n$ and $(L',L'') \in \mcal O_\mak a$.
Then $L_i=L_i'$ for $i \in [1,r]$ and $|L_0/L_0\cap L_0'|=1$.
Moreover, we have $L\in \mcal X_{n,d}^{\mak d,2}$.
Assume that
\begin{align*}
L_0\cap L_0'\cap L_j''=L_0'\cap L_j'', \ {\rm for}\ j<p; \ \  L_0\cap L_0'\cap L_j''\neq L_0'\cap L_j'', \ {\rm for}\ j\geq p.
\end{align*}
Then $(L,L'')\in \mcal O_{(A+E_\theta^{1,p}+E_\theta^{0,p},-)}$,
and we have the following:
\begin{equation*}
e_\mak n\ast e_\mak a=\sum_{p\in \mbb Z, a_{0,p}\geq 1}G_{\mak a,\mak n;p}
e_{(A+E_\theta^{1,p}-E_{\theta}^{0,p}, -)}.
\end{equation*}
Now, we calculate the $G_{\mak a,\mak n;p}$.
Let $L=(L_i)_{i\in \mbb Z}, L'=(L_i')_{i\in \mbb Z} \in \mcal X_{n,d}^\mak d$
such that $(L, L') \in \mcal O_{(A+E_\theta^{1,p}+E_\theta^{0,p}, -)}$.
Consider the set $Z_p$ of filtrations of lattices $L''=(L_i'')_{i\in \mbb Z}\in \mcal X_{n,d}^\mak d$ such that $(L,L'')\in \mcal O_\mak n$ and $(L'', L') \in \mcal O_\mak a$.
Then  $L_i=L_i''$ for $i \in [1,r]$ and $|L_0/L_0\cap L_0''|=1$.
Moreover, we have
\begin{align*}
L_0\cap L_0''\cap L_j'=L_0\cap L_j', \ {\rm for}\ j<1-p;  \ \ L_0\cap L_0''\cap L_j'\neq L_0\cap L_j', \ {\rm for} \ j\geq 1-p.
\end{align*}
If $p\leq 0$, then
\begin{align*}
\sharp(Z_p)&=\sharp\{L_0'' \mid L_{-1}+L_0\cap L_{-p}'\subset L_0''\cap L_0\}-\sharp\{L_0'' \mid L_{-1}+L_0\cap L_{1-p}'\subset L_0''\cap L_0\}\\
&=(q-1)^{-1}(q^{|L_0/L_{-1}+L_0\cap L_{-p}'|}-q^{|L_0/L_{-1}+L_0\cap L_{1-p}'|})\\
&=(q-1)^{-1}(q^{\sum_{j\geq 1-p}a_{0,j}+1}-q^{\sum_{j> 1-p}a_{0,j}}).
\end{align*}
If $p> 0$, then
\begin{align*}
\sharp(Z_p)&=\sharp\{L_0'' \mid L_{-1}+L_0\cap L_{-p}'\subset L_0''\cap L_0\}-\sharp\{L_0''\mid L_{-1}+L_0\cap L_{1-p}'\subset L_0''\cap L_0\}\\
&=(q-1)^{-1}(q^{|L_0/L_{-1}+L_0\cap L_{-p}'|}-q^{|L_0/L_{-1}+L_0\cap L_{1-p}'|})\\
&=(q-1)^{-1}(q^{\sum_{j\geq 1-p}a_{0,j}}-q^{\sum_{j> 1-p}a_{0,j}-1}).
\end{align*}
This implies case $(c)$ is verified.
The case $(c')$ is the counterpart of the case $(c)$, and hence we omit it.
\end{proof}

Given two integers $a,b$ and $b\geq 0$.
We set
\begin{align*}
[a]=\frac{v^{2a}-1}{v^{2}-1}, \ \ \ {\rm and} \ \ \  [a,b]=\prod_{i \in [1,b]}\frac{v^{2(a-i+1)}-1}{v^{2i}-1}.
\end{align*}
We define a bar involution $'-'$ on $\mcal A$ by $\overline{v}=v^{-1}$.

By an induction process, we now generalize Lemma \ref{formulas} to some multiplication formulas by $``{\rm divided \ powers}"$ of Chevalley generators.
The proof involves lengthy mechanical computations and is hence skipped.
\begin{proposition}\label{divided formula}
Let $\mak a=(A, \alpha) \in \Xi_\mak d$.
Assume that $R \in \mbb N$ and $h \in [1,r-1]$.

$(a)$ Assume that $\mak b=(B, \alpha) \in \Xi_\mak d$ such that $B-RE_\theta^{h,h+1}$ is a diagonal matrix and ${\rm co}(\mak b)={\rm ro}(\mak a)$.
Then
\begin{equation}\label{divided formula 1}
e_\mak b\ast e_\mak a=\sum_{t}v^{2\sum_{j>u}a_{h,j}t_u}\prod_{u\in \mbb Z}
[a_{h,u}+t_u, t_u] e_{(A+\sum_{u\in \mbb Z}t_u(E_\theta^{h,u}-E_\theta^{h+1,u}), \alpha)},
\end{equation}
where the sum runs over all sequences $t=(t_u)_{u\in \mbb Z} \in \mbb N ^{\mbb Z}$ such that $0\leq t_u \leq a_{h+1,u}$ and $\sum_{u\in \mbb Z}t_u=R$.
Moreover,
\begin{align*}
[\mak b]\ast[\mak a]=\sum_{t}v^{\beta(t)}\prod_{u \in \mbb Z}\overline{[a_{h,u}+t_u, t_u]}[(A+\sum_{u\in \mbb Z}(E_\theta^{h,u}-E_\theta^{h+1,u}), \alpha)],
\end{align*}
where the sum is over $t$ as in \eqref{divided formula 1} and
\begin{align*}
\beta(t)=\sum_{j\geq u}a_{h,j}t_u-\sum_{j>u}a_{h+1,j}t_u+\sum_{u<u'}t_ut_{u'}.
\end{align*}
$(b)$ Assume that $\mak c=(C, \alpha) \in \Xi_\mak d$ such that $C-RE_\theta^{h+1,h}$ is a diagonal matrix and ${\rm co}(\mak c)={\rm ro}(\mak a)$.
Then
\begin{equation} \label{divided formula 2}
e_\mak c\ast e_\mak a=\sum_{t}v^{2\sum_{j<u}a_{h+1,j}t_u}\prod_{u\in \mbb Z}
[a_{h+1,u}+t_u, t_u]e_{(A+\sum_{u \in \mbb Z}t_u(E_\theta^{h+1,u}-
E_\theta^{h,u}), \alpha)},
\end{equation}
where the sum runs over all sequences $t=(t_u)_{u\in \mbb Z} \in \mbb N^{\mbb Z}$ such that $0\leq t_u \leq a_{h,u}$ and $\sum_{u\in \mbb Z}t_u=R$.
Moreover,
\begin{align*}
[\mak c]\ast[\mak a]=\sum_{t}v^{\beta'(t)}\prod_{u \in \mbb Z}
\overline{[a_{h+1,u}+t_u, t_u]}[(A-\sum_{u \in \mbb Z}
t_u(E_\theta^{h,u}-E_\theta^{h+1,u}), \alpha)],
 \end{align*}
where the sum is over $t$ as in \eqref{divided formula 2} and
\begin{align*}
\beta'(t)=\sum_{j\leq u}a_{h+1,j}t_u-\sum_{j<u}a_{h,j}t_u+\sum_{u<u'}t_ut_{u'}.
\end{align*}
$(c)$  Assume that $\mak n=(D, -\alpha) \in \Xi_{\mak d}$ satisfies $D-E_{\theta}^{0,1}$ is a diagonal matrix and ${\rm co}(\mak n)={\rm ro}(\mak a)$.
 Then
\begin{equation}
\begin{split}
[\mak n]\ast[\mak a]&=\sum_{p\leq 0,a_{0,p}\geq 1}v^{\sum_{j\leq p}
a_{1,j}-\sum_{j<p}a_{0,j}}\overline{[a_{1,p}+1]}
[(A-E_{\theta}^{0,p}+E_{\theta}^{1,p},-\alpha)]\\
&+\sum_{p> 1,a_{0,p}\geq 1}v^{\sum_{j\geq p}
a_{0,j}-\sum_{j>p}a_{1,j}-1}\overline{[a_{1,p}+1]}
[(A-E_{\theta}^{0,p}+E_{\theta}^{1,p},-\alpha)].
\end{split}
\end{equation}

$(c')$  Assume that $\mak n'=(D', -\alpha) \in \Xi_\mak d$ satisfies $D'-E_{\theta}^{r,r+1}$ is a diagonal matrix and ${\rm co}(\mak n')={\rm ro}(\mak a)$.
Then
\begin{equation}
\begin{split}
[\mak n']\ast[\mak a]&=\sum_{p\leq r,a_{r,p}\geq 1}v^{\sum_{j\leq p}
a_{r+1,j}-\sum_{j<p}a_{r,j}}\overline{[a_{r+1,p}+1]}
[(A-E_{\theta}^{r,p}+E_{\theta}^{r+1,p},-\alpha)]\\
&+\sum_{p >r,a_{r,p}\geq 1}v^{\sum_{j\geq p}
a_{r,j}-\sum_{j>p}a_{r+1,j}-1}\overline{[a_{r+1,p}+1]}
[(A-E_{\theta}^{r,p}+E_{\theta}^{r+1,p},-\alpha)].
\end{split}
\end{equation}
\end{proposition}

\subsection{The leading term}
In the expression $``[\mak a]+{\rm lower \ terms}"$ below, the $``{\rm lower \ terms}"$ represents a linear combination of elements strictly less than $\mak a$ with respect to the partial order $``\preceq"$.
\begin{lemma}\label{ZXY}
Assume that $h \in [1,r-1]$, and $R$ is a positive integer.

$(a)$ Assume that $\mak b=(B,\alpha) \in \Xi_\mak d$ and $B-RE_\theta^{h,h+1}$ is diagonal.
Suppose that $\mak a=(A,\alpha) \in \Xi_\mak d$ such that ${\rm co}(\mak b)={\rm ro}(\mak a)$ and satisfies the following conditions:
\begin{equation}
a_{h,j}=0, \ \forall j\geq k; \ a_{h+1,k}=R, \ a_{h+1,j}=0, \ \forall j>k, \ {\rm for} \ k\geq h+1.
\end{equation}
Then we have
\begin{align*}
[\mak b]\ast [\mak a]=[(A+R(E_{\theta}^{h,k}-E_\theta^{h+1,k}),\alpha)]+{\rm lower \ terms}.
\end{align*}

$(b)$ Assume that $\mak c=(C,\alpha) \in \Xi_\mak d$ and $C-RE_\theta^{h+1,h}$ is diagonal.
Suppose that $\mak a=(A,\alpha) \in \Xi_\mak d$ such that ${\rm co}(\mak c)={\rm ro}(\mak a)$ and satisfies the following conditions:
\begin{equation}
a_{h,j}=0,\  \forall j< k; \ a_{h,k}=R, \ a_{h+1,j}=0, \ \forall j\leq k, \ {\rm for} \ k\leq h.
\end{equation}
Then we have
\begin{align*}
[\mak c]\ast [\mak a]=[(A-R(E_{\theta}^{h,k}-E_\theta^{h+1,k}),\alpha)]+{\rm lower \ terms}.
\end{align*}

$(c_1)$ Assume that $\mak n=(D, \alpha) \in \Xi_\mak d$, if $R$ is even and $D-RE_\theta^{0,1}$ is diagonal.
Suppose that $\mak a=(A,\alpha) \in \Xi_\mak d$ such that ${\rm co}(\mak n)={\rm ro}(\mak a)$ and satisfies the following conditions:
\begin{equation}\label{chen}
a_{0,j}=0, \ \forall j\geq k; \ a_{1,k}=R, \ a_{1,j}=0, \ \forall j>k, \ {\rm for}\ k\geq 1.
\end{equation}
Then we have
\begin{align*}
[\mak n]\ast [\mak a]=[(A+R(E_{\theta}^{0,k}-E_\theta^{1,k}),\alpha)]+{\rm lower \ terms}.
\end{align*}

$(c_1')$ Assume that $\mak n=(D,\alpha) \in \Xi_\mak d$, if $R$ is odd and $D-RE_\theta^{0,1}$ is diagonal.
Suppose that $\mak a=(A,-\alpha) \in \Xi_\mak d$ such that ${\rm co}(\mak n)={\rm ro}(\mak a)$ and satisfies the following conditions:
\begin{equation}
a_{0,j}=0, \ \forall j\geq k; \ a_{1,k}=R, \ a_{1,j}=0,\  \forall j>k, \ {\rm for}\ k\geq 1.
\end{equation}
Then we have
\begin{align*}
[\mak n]\ast [\mak a]=[(A+R(E_{\theta}^{0,k}-E_\theta^{1,k}),\alpha)]+{\rm lower \ terms}.
\end{align*}

$(c_2)$ Assume that $\mak n'=(D',\alpha) \in \Xi_\mak d$, if $R$ is even and $D-RE_\theta^{r,r+1}$ is diagonal.
Suppose that $\mak a=(A,\alpha) \in \Xi_\mak d$ such that ${\rm co}(\mak n')={\rm ro}(\mak a)$ and satisfies the following conditions:
\begin{equation}
a_{r,j}=0, \ \forall j\geq k;\  a_{r+1,k}=R,\  a_{r+1,j}=0,\  \forall j>k,  \ {\rm for}\ k\geq r+1.
\end{equation}
Then we have
\begin{align*}
[\mak n']\ast [\mak a]=[(A+R(E_{\theta}^{r,k}-E_\theta^{r+1,k}),\alpha)]+{\rm lower \ terms}.
\end{align*}

$(c_2')$ Assume that $\mak n'=(D',\alpha) \in \Xi_\mak d$, if $R$ is odd and $D-RE_\theta^{r,r+1}$ is diagonal.
Suppose that $\mak a=(A,-\alpha) \in \Xi_\mak d$ such that ${\rm co}(\mak n')={\rm ro}(\mak a)$ and satisfies the following conditions:
\begin{equation}
a_{r,j}=0,\  \forall j\geq k; \ a_{r+1,k}=R, \ a_{r+1,j}=0, \ \forall j>k,  \ {\rm for}\ k\geq r+1.
\end{equation}
Then we have
\begin{align*}
[\mak n']\ast [\mak a]=[(A+R(E_{\theta}^{r,k}-E_\theta^{r+1,k}),\alpha)]+{\rm lower \ terms}.
\end{align*}
\end{lemma}
\begin{proof}
The proof of cases $(a)$ and $(b)$ is the same as the one for \cite[Lemma 4.4.1]{FLLLW20}.
For the case of $(c_1)$, we set that
\begin{align*}
M=A+R(E_\theta^{1,k}-E_\theta^{0,k}), \ \   M'=A+\sum_{u \in \mbb Z}
t_u(E_\theta^{1,u}-E_\theta^{0,u}),
\end{align*}
with $\sum_{u \in \mbb Z}t_u= R'$ for some even integer $R'\leq R$ and $t_u=0$ unless $u\in (1-k,k]$.
The $(r,s)$-th entry $m_{rs}$ of $M$ is
\begin{align*}
m_{rs}=a_{rs}+R\sum_{l \in \mbb Z}
\delta_{s,k+ln}(\delta_{r,ln}-\delta_{r,1+ln})-R\sum_{l \in \mbb Z}
\delta_{s,1-k+ln}(\delta_{r,ln}-\delta_{r,1+ln}).
\end{align*}
Note that
\begin{align}
\label{M-c}
 \sum_{r \leq i, s\geq j}&R \sum_{l \in \mbb Z}
 \delta_{s,k+ln}(\delta_{r,ln}-\delta_{r,1+ln}) \\
&=
\begin{cases}
R, & \mbox{if} \ i=l_1n, j\leq k+l_1n, \ \mbox{for some} \ l_1 \in \mbb Z,\\
0, & \mbox{otherwise}.
\end{cases}
\nonumber \\
\label{M-d}
\sum_{r \leq i, s\geq j} & -R\sum_{l \in \mbb Z}
\delta_{s,1-k+ln}(\delta_{r,ln}-\delta_{r,1+ln}) \\
&=
\begin{cases}
-R , & \mbox{if} \ i=l_1n, j\leq 1-k+l_1n, {\rm for \ some}\ l_1\in \mbb Z; \\
0, & \mbox{otherwise}.
\end{cases}
\nonumber
\end{align}
For $M'$, the $(r,s)$-th entry $m_{rs}'$ of $M'$ is
\begin{align*}
m_{rs}'=a_{rs}+\sum_{l \in \mbb Z}t_{s-ln}(\delta_{r,ln}-\delta_{r,1+ln})-\sum_{l \in \mbb Z}t_{1-s+ln}(\delta_{r,ln}-\delta_{r,1+ln}).
\end{align*}
Similarly, we have
\begin{align}
\label{M-a}
 \sum_{r \leq i, s\geq j}&\sum_{l \in \mbb Z}t_{s-ln}(\delta_{r,ln}-\delta_{r,1+ln})\\
&=
\begin{cases}
\sum_{s\geq j-l_1n}t_s, & \mbox{if} \ i=l_1n, \ \mbox{for some} \ l_1 \in \mbb Z,\\
0, & \mbox{otherwise}.
\end{cases}
\nonumber \\
\label{M-b}
\sum_{r \leq i, s\geq j} & \sum_{l \in \mbb Z}
-t_{1-s+ln}(\delta_{r,ln}-\delta_{r,1+ln}) \\
&=
\begin{cases}
-\sum_{s\leq 1+l_1n-j}t_s , & \mbox{if} \ i=l_1n, {\rm for \ some}\ l_1\in \mbb Z; \\
0, & \mbox{otherwise}.
\end{cases}
\nonumber
\end{align}
 We shall show that $(M',\alpha) \preceq (M,\alpha)$ when $\mak a$ is subject to the condition \eqref{chen}.
It suffices to show \eqref{M-c} $\geq$ \eqref{M-a} and \eqref{M-d} $\geq$ \eqref{M-b} when $i<j$.
Since \eqref{M-d} is always equal to $0$ when $i<j$, we have \eqref{M-d} $\geq$ \eqref{M-b}.
If $i=l_1n< j\leq k+l_1n$, then $\sum_{s\geq j-l_1n}t_s\leq R$.
On the other hand, if $i=l_1n$ and $j>k+l_1n$, we see that $\sum_{s\geq j-l_1n}t_s\geq \sum_{s\geq k+1}t_s=0$.
Thus, we have $(M',\alpha) \preceq (M,\alpha)$.

We now verify the coefficient of $(M,\alpha)$ is $1$.
We set $\mak b'=(B',\alpha)$ and $\mak b''=(B',-\alpha)$ such that ${\rm co}(\mak b')={\rm ro}(\mak a)$ and $B'-E_\theta^{0,1}$ is diagonal, then
\begin{equation*}
(e_{\mak b'}\ast e_{\mak b''})^{\frac{R}{2}}=[2]\cdots[R]
e_{\mak n}+\sum g_{\mak b',\mak b'';\mak t}e_{\mak t},
\end{equation*}
and
\begin{equation*}
(e_{\mak b'}\ast e_{\mak b''})^{\frac{R}{2}}\ast e_{\mak a}=[2]\cdots[R]
e_{(M,\alpha)}+{\rm lower \ terms},
\end{equation*}
where $\mak t=(T,\alpha)$ and $T-R'E_\theta^{0,1}$ is diagonal for some even integers $R'< R$.
By induction, we have
\begin{equation*}
e_{\mak t}\ast e_{\mak a}=\sum G_{\mak a',v}e_{\mak a'},
\end{equation*}
where $\mak a'$ is the form of $(M',\alpha)$ and $(M',\alpha)\prec (M,\alpha)$.
Then we have
\begin{equation*}
e_{\mak n}\ast e_{\mak a}=e_{(M,\alpha)}+{\rm lower \ terms}.
\end{equation*}
According to Lemma \ref{matric}, we see that
\begin{equation*}
\begin{split}
d(\mak n)&=Rd_{0,0}+\frac{R^2-R}{2},\\
d((M,\alpha))-d(\mak a)&=R\sum_{1-k\leq j<k}a_{0,j}-\frac{R^2+R}{2},
\end{split}
\end{equation*}
where $d_{0,0}$ is the $(0,0)$-th entry of the matrix $D$.
Since $R+d_{0,0}=\sum_{j\in \mbb Z}a_{0,j}$ and $a_{0,j}=0$ unless $1-k\leq j< k$, then
\begin{equation}
d((M,\alpha))-d(\mak a)=d(\mak n),
\end{equation}
which is as required.
The cases $(c_2), (c_1')$ and $(c_2')$ are the counterparts with $(c_1)$, and hence we skip it.
\end{proof}

\begin{lemma}\label{CQY}
Assume that $h \in [1,r-1]$, and $R$ is a positive integer.

$(a)$ Assume that $\mak b=(B,\alpha) \in \Xi_\mak d$ such that $B-RE_\theta^{h,h+1}$ is diagonal and $R=R_0+\cdots+R_l$.
Suppose that $\mak a=(A,\alpha) \in \Xi_\mak d$ such that ${\rm co}(\mak b)={\rm ro}(\mak a)$ and satisfies the following conditions:
\begin{align*}
a_{h,j}=0, \ \forall \ j\geq k;\  a_{h+1,k}\geq R_0,\  a_{h+1,k+i}=R_i,\  i\in [1,l], \ a_{h+1,j}=0, \ \forall j>k+l.
\end{align*}
Then we have
\begin{align*}
[\mak b]\ast[\mak a]
=[(A+\sum_{i=0}^lR_i(E_{\theta}^{h,k+i}-E_\theta^{h+1,k+i}),\alpha)]+
{\rm lower \ terms}.
\end{align*}

$(b)$ Assume that $\mak c=(C,\alpha) \in \Xi_\mak d$ such that $C-RE_\theta^{h+1,h}$ is diagonal and $R=R_0+\cdots+R_l$.
Suppose that $\mak a=(A,\alpha) \in \Xi_\mak d$ such that ${\rm co}(\mak c)={\rm ro}(\mak a)$ and satisfies the following conditions:
\begin{align*}
a_{h+1,j}=0, \ \forall \ j\leq k+l;\  a_{h,k+i}=R_i, i\in [0,l-1],\  a_{h,k+l}\geq R_l,\  \ a_{h,j}=0, \ \forall j<k.
\end{align*}
Then we have
\begin{align*}
[\mak c]\ast [\mak a]
=[(A-\sum_{i=0}^lR_i(E_{\theta}^{h,k+i}-E_\theta^{h+1,k+i}),\alpha)]+
{\rm lower \ terms}.
\end{align*}

$(c_1)$ Assume that $\mak n=(D,\alpha) \in \Xi_\mak d$, if $R$ is even and $D-RE_\theta^{0,1}$ is diagonal, we assume that $R=R_0+\cdots+R_l$.
Suppose that $\mak a=(A,\alpha) \in \Xi_\mak d$ such that ${\rm co}(\mak n)={\rm ro}(\mak a)$ and satisfies the following conditions:
\begin{align*}
a_{0,j}=0, \forall j\geq k\geq 1; a_{1,k}\geq R_0, a_{1,k+i}=R_i, \ i\in [1,l] ,\ a_{1,j}=0, \forall j>k+l.
\end{align*}
Then we have
\begin{align*}
[\mak n]\ast [\mak a] =[(A+\sum_{i=0}^lR_i(E_{\theta}^{0,k+i}-E_\theta^{1,k+i}),\alpha)]+{\rm lower \ terms}.
\end{align*}

$(c_1')$ Assume that $\mak n=(D,\alpha) \in \Xi_\mak d$, if $R$ is odd and $D-RE_\theta^{0,1}$ is diagonal, we assume that $R=R_0+\cdots+R_l$.
Suppose that $\mak a=(A,-\alpha) \in \Xi_\mak d$ such that ${\rm co}(\mak n)={\rm ro}(\mak a)$ and satisfies the following conditions:
\begin{align*}
a_{0,j}=0, \forall j\geq k\geq 1; a_{1,k}\geq R_0, \ a_{1,k+i}=R_i,\ i \in [1,l] \ a_{1,j}=0, \forall j>k+l.
\end{align*}
Then we have
\begin{align*}
[\mak n]\ast [\mak a] =[(A+\sum_{i=0}^lR_i(E_{\theta}^{0,k+i}-E_\theta^{1,k+i}),\alpha)]+{\rm lower \ terms}.
\end{align*}

$(c_2)$ Assume that $\mak n'=(D',\alpha) \in \Xi_\mak d$, if $R$ is even and $D-RE_\theta^{r,r+1}$ is diagonal,
we assume that $R=R_0+\cdots+R_l$.
Suppose that $\mak a=(A,\alpha) \in \Xi_\mak d$ such that ${\rm co}(\mak n')={\rm ro}(\mak a)$ and satisfies the following conditions:
\begin{align*}
a_{r,j}=0, \forall j\geq k\geq r+1; a_{r+1,k}\geq R_0, \ a_{r+1,k+i}=R_i,\ i \in [1,l] \ a_{r+1,j}=0, \forall j>k+l.
\end{align*}
Then we have
\begin{align*}
[\mak n']\ast [\mak a] =[(A+\sum_{i=0}^lR_i(E_{\theta}^{r,k+i}-E_\theta^{r+1,k+i}),\alpha)]+{\rm lower \ terms}.
\end{align*}

$(c_2')$ Assume that $\mak n'=(D',\alpha) \in \Xi_\mak d$,
if $R$ is odd and $D-RE_\theta^{r,r+1}$ is diagonal,
we assume that $R=R_0+\cdots+R_l$.
Suppose that $\mak a=(A,-\alpha) \in \Xi_\mak d$ such that ${\rm co}(\mak n')={\rm ro}(\mak a)$ and satisfies the following conditions:
\begin{align*}
a_{r,j}=0, \forall j\geq k\geq r+1; a_{r+1,k}\geq R_0, \ a_{r+1,k+i}=R_i,\ i \in [1,l] \ a_{r+1,j}=0, \forall j>k+l.
\end{align*}
Then we have
\begin{align*}
[\mak n']\ast [\mak a] =[(A+\sum_{i=0}^lR_i(E_{\theta}^{r,k+i}-E_\theta^{r+1,k+i}),\alpha)]+{\rm lower \ terms}.
\end{align*}
\end{lemma}
\begin{proof}
The proofs of cases $(a)$ and $ (b)$ is clear by the results of Lemma \ref{ZXY} and Proposition \ref{divided formula}.
We now prove the case $(c_1)$.
Let $\mak y=(A+\sum_{i=0}^lR_i(E_{\theta}^{0,k+i}-E_\theta^{1,k+i}),\alpha)$.
By Lemma \ref{ZXY},
we get that $[\mak y]$ is the leading term, we only to consider the coefficient of it.
By induction, we have
\begin{equation*}
e_\mak n\ast e_\mak a=e_\mak y+{\rm lower \ terms}.
\end{equation*}
 So it is left to verify $d(\mak y)-d(\mak a)=d(\mak n)$.
 We see that
 \begin{align*}
 d(\mak y)-d(\mak a)&=
 \sum_{u \in [0,l]} t_u(\sum_{1-k\leq j<k}a_{0,j}+\sum_{u'<u}t_{u'})-\sum_{u \in [0,l]}\frac{t_u^2+t_u}{2},\\
 d(\mak n)&=Rd_{0,0}+\frac{R^2-R}{2}.
 \end{align*}
 Moreover, we have $\sum_{1-k\leq j<k}a_{0,j}=d_{0,0}+t_0$, which is as required.
\end{proof}

\section{The Lusztig algebra $\mbf U_{n,d}^\mak d$ and its basis}
In this section, we show that there is a subalgebra $\widetilde{\mbf S}_{n,d}^\mak d$ of Schur algebra and formulate a coideal algebra type structure which involves $\widetilde{\mbf S}_{n,d}^\mak d$ and Shcur algebra of affine type $\mrm A$, and its behavior on the chevalley generators.
The canonical basis of the Lusztig algebra are shown to be compatible under the inclusion $\mbf U_{n,d}^\mak d \subset \mbf S_{n,d}^\mak d$.
\subsection{The Lusztig algebra $\mbf U_{n,d}^\mak d$}
Let $\mbf U_{n,d}^\mak d$ be the subalgebra of $\mbf S_{n,d}^\mak d$
generated by all elements $[\mak b]$ such that $\mak b=(B,\alpha)$
 satisfies $B$ or $B-E_\theta^{h,h+1}$ is diagonal for various $h\in \mbb Z$.
Let $\mbf U_{n,d;\mcal A}^\mak d$ denote the $\mcal A$-subalgebra of $\mbf S_{n,d,\mcal A}^\mak d$
generated by all elements $[\mak b]$ such that $\mak b=(B, \alpha)$
satisfies $B$ or $ B-RE_\theta^{h,h+1}$ is diagonal for various $h \in \mbb Z$
and $R \in \mbb N$.
\begin{definition}
The algebra $\mbf U_{n,d}^\mak d$ is called the Lusztig algebra (of affine type $\mrm D$).
\end{definition}
The definition of aperiodic
 singed matrix is same as the definition of type $\mrm A$ in \cite{Lu99}.
\begin{definition}
A signed matrix $\mak a=(A, \alpha)$ is called if $A$ is aperiodic.
\end{definition}
Denote by $\Xi_{n,d}^{ap}$ and $\Xi_\mak d^{ap}$ the sets of aperiodic matrices in $\Xi_{n,d}$ and aperiodic signed matrices in $\Xi_\mak d$, respecetively.
A product of standard basis elements $[\mak b_1]\ast \cdots \ast[\mak b_m]$ is called an aperiodic monomial
if for each $i \in [1,m]$, such that $B_i-RE_{\theta}^{j,j+1}$ is diagonal for some $R \in \mbb N$ and $j \in \mbb Z$.

The following aperiodic monomial is an analogue of $\mbf U_{n,d}$ (Lusztig algebra of affine type $\mrm A$) in \cite{Lu99}.
\begin{lemma}\label{chen121212}
For any $\mak a \in \Xi_\mak d^{ap}$,
there exists an aperiodic monomial $\zeta_\mak a=[\mak b_1]\ast \cdots \ast[\mak b_m]$ such that $[\mak b_1]\ast  \cdots \ast[\mak b_m]=[\mak a]+\ {\rm lower \ terms}$.
\end{lemma}
\begin{proof}
Given $\mak a=(A, \alpha) \in \Xi_\mak d$,
we set
\begin{align*}
f_{k;s,t}(\mak a)=(A-\sum_{s\leq j\leq t}a_{k-1,j}(E_\theta^{k-1,j}-E_\theta^{k,j}), \alpha')\in \Xi_\mak d,
\end{align*}
where $\alpha'=\alpha$ unless $k \equiv 1 \ {\rm mod}\ r$ and $\sum_{s\leq j\leq t}a_{k-1,j}$ is even.

Let $\Psi(\mak a)=\sum_{i\in [1,n], i<j}|j-i|a_{ij}$.
It is apparent that $\Psi({f_{k;s,t}(\mak a)})\leq \Psi(\mak a)$
for all $k,s$ and $t$ with $k\leq s\leq t$,
and the equality holds if and only if
\begin{align*}
a_{k-1,j}=0, \ \forall s\leq j\leq t.
\end{align*}
We prove it by induction on $\Psi(\mak a)$.
If $\Psi(\mak a)=0$, then $A$ is diagonal.

Assume that $\Psi(\mak a)>0$ and there exist aperiodic monomials for all $\mak a'$ such that $\Psi(\mak a')<\Psi(\mak a)$.
Set $m={\rm min}\{l \in \mbb N \mid a_{ij}=0 \ {\rm for \ all}\ j-i>l\}$.
Since $\mak a$ is aperiodic, there exists $a_{k,k+m}=0$ and $a_{k-1,k+m-1}\neq 0$ for some $k\in \mbb Z$.

Let $u={\rm max}\{s\leq k+m-1 \mid f_{k;s,k+m-1}(\mak a)\}$.
It is clear that $u\geq k$ and $a_{k,l}=0$ for all $l>u$.
Otherwise, there exists $j>u$ such that $a_{kj}\neq 0$.
Then $f_{k;j,k+m-1}(\mak a)$ is aperiodic, it is absurd.

We set $\mak b=(B,\alpha) \in \Xi_\mak d$ such that ${\rm ro}(f_{k;s,t}(\mak a))={\rm co}(\mak b)$ and $B-\sum_{u}^{k+m-1}a_{k-1,j}E_\theta^{k-1,k}$ is diagonal.
 By the results of Lemma \ref{CQY}, we have
\begin{align*}
[\mak b]\ast[f_{k;s,t}(\mak a)]=[\mak a]+{\rm lower \ terms}.
\end{align*}
 By induction on $\Psi(\mak a)$, we complete the proof.
\end{proof}

 We shall use the notation $\mcal L \overset{1}{\subset}\mcal L'$ to denote that $ \mcal L \subset \mcal L'$ and $|\mcal L'/\mcal L|=1$.
 For $1\leq i <r$ and $1\leq a \leq r$, we define the following functions, for any $L =(L_i)_{i \in \mbb Z}, L'=(L_i')_{i \in \mbb Z}\in \mcal X_{n,d}^\mak d$
\begin{align*}
\mbf E_i(L,L')&=\left\{\begin{array}{ll}
   v^{-|L_{i+1}'/L_{i}'|}, & {\rm if}\ L_i\overset{1}{\subset}L'_i, L_j=L_j',
   \forall j \in [0,r]\backslash\{i\};\\
   0, &{\rm oterwise }. \\
 \end{array}
 \right.\\
 \mbf F_i(L,L')&=\left\{\begin{array}{ll}
   v^{-|L_{i}'/L_{i-1}'|}, & {\rm if}\ L_i'\overset{1}{\subset}L_i, V_j=V'_j, \forall j \in [1,r]\backslash\{i\};\\
   0, &{\rm oterwise }. \\
 \end{array}
 \right.\\
 \mbf H_a^{\pm}(L,L')&=\left\{\begin{array}{ll}
   v^{\pm|L_{a}'/L_{a-1}'|}, & {\rm if}\ L=L';\\
   0, &{\rm oterwise }. \\
 \end{array}
 \right.\\
 \mbf T_0(L,L')&=\left\{\begin{array}{ll}
   v^{1-|L_{0}'/L_{-1}'|}, & {\rm if}\ L_0\cap L_0'\overset{1}{\subset}L_0, L_j=L_j', \forall j \in [1,r];\\
   0, &{\rm oterwise }. \\
 \end{array}
 \right.\\
 \mbf T_r(L,L')&=\left\{\begin{array}{ll}
   v^{1-|L_{r}'/L_{r-1}'|}, & {\rm if}\ L_r\cap L_r'\overset{1}{\subset}L_r, L_j=L_j', \forall j \in [0,r-1];\\
   0, &{\rm oterwise }. \\
 \end{array}
 \right.\\
 \mbf J_+(L,L')&=\left\{\begin{array}{ll}
   1, & {\rm if}\ L=L', \  {\rm and} \ L \in \mathcal{X}_{n,d}^{\mathfrak{d},1};\\
   0, &{\rm oterwise }. \\
 \end{array}
 \right.\\
 \mbf J_-(L,L')&=\left\{\begin{array}{ll}
   1, & {\rm if}\ L=L', \ {\rm and} \ L \in \mathcal{X}_{n,d}^{\mathfrak{d},2};\\
   0, &{\rm oterwise }. \\
 \end{array}
 \right.\\
 \mbf K_i&=\mbf H_{i+1}\mbf H_i^{-1}.
 \end{align*}

 \begin{proposition}
 The elements $\mbf E_i, \mbf F_i, \mbf K_i$ for $i\in [1,r-1]$ and $\mbf J_\pm, \mbf T_0, \mbf T_r$ satisfy the following relations in $\mbf U_{n,d}^\mak d$.
\begin{align*}
&\mbf J_\pm \mbf J_\pm =\mbf J_\pm,\ \  \mbf J_\pm \mbf J_\mp=0, \\
&\mbf T_0\mbf T_r=\mbf T_r\mbf T_0, \ \ \mbf K_i\mbf K_j=\mbf K_j\mbf K_i,\\
&\mbf J_\pm\mbf E_i=\mbf E_i\mbf J_\pm, \ \ \mbf J_\pm\mbf F_i=\mbf F_i\mbf J_\pm, \\
&\mbf J_\pm\mbf T_0=\mbf T_0\mbf J_\mp, \ \ \mbf J_\pm\mbf T_r=\mbf T_r\mbf J_\mp, \\
&\mbf K_i\mbf E_j=v^{2\delta_{ij}-\delta_{i,j+1}-\delta_{i+1,j}}\mbf E_j\mbf K_i, \\
&\mbf K_i\mbf F_j=v^{-2\delta_{ij}+\delta_{i,j+1}+\delta_{i+1,j}}\mbf F_j\mbf K_i, \\
&\mbf K_i\mbf T_0=\mbf T_0\mbf K_i, \ \ \mbf T_r\mbf K_i=\mbf K_i\mbf T_r, \\
&\mbf E_i\mbf F_j-\mbf F_j\mbf E_i=\delta_{ij}\frac{\mbf K_i-\mbf K_i^{-1}}{v-v^{-1}},\\
&\mbf E_i\mbf E_j=\mbf E_j\mbf E_i, \ \ \mbf F_i\mbf F_j=\mbf F_j\mbf F_i,\ \ {\rm if} \ |i-j|>1,\\
&\mbf E_i^2\mbf E_j+\mbf E_j\mbf E_i^2=(v+v^{-1})\mbf E_i\mbf E_j\mbf E_i, \ \ {\rm if} \ |i-j|=1, \\
&\mbf F_i^2\mbf F_j+\mbf F_j\mbf F_i^2=(v+v^{-1})\mbf F_i
\mbf F_j\mbf F_i,\ \ {\rm if} \ |i-j|=1, \\
&\mbf T_0\mbf E_i=\mbf E_i\mbf T_0, \ \ \mbf T_0\mbf F_i=\mbf F_i\mbf T_0, \ \ {\rm if} \ i\in [2,r-1] \\
&\mbf T_r\mbf E_i=\mbf E_i\mbf T_r, \ \ \mbf T_r\mbf F_i=\mbf F_i\mbf T_r, \ \ {\rm if} \ i \in [1,r-2] \\
&\mbf E_1\mbf T_0^2+\mbf T_0^2\mbf E_1=(v+v^{-1})\mbf T_0 \mbf E_1\mbf T_0+\mbf E_1, \\
&\mbf T_0\mbf E_1^2+\mbf E_1^2\mbf T_0=(v+v^{-1})\mbf E_1 \mbf T_0\mbf E_1, \\
&\mbf F_1\mbf T_0^2+\mbf T_0^2\mbf F_1=(v+v^{-1})\mbf T_0
\mbf F_1\mbf T_0+\mbf F_1, \\
&\mbf T_0\mbf F_1^2+\mbf F_1^2\mbf T_0=(v+v^{-1})\mbf F_1 \mbf T_0\mbf F_1, \\
&\mbf E_{r-1}\mbf T_r^2+\mbf T_r^2\mbf E_{r-1}=(v+v^{-1})\mbf T_r
\mbf E_{r-1}\mbf T_r+\mbf E_{r-1},\\
&\mbf T_r\mbf E_{r-1}^2+\mbf E_{r-1}^2\mbf T_r=(v+v^{-1})
\mbf E_{r-1}\mbf T_r\mbf E_{r-1},  \\
&\mbf F_{r-1}\mbf T_r^2+\mbf T_r^2\mbf F_{r-1}=(v+v^{-1})\mbf T_r
\mbf F_{r-1}\mbf T_r+\mbf F_{r-1}, \\
&\mbf T_r\mbf F_{r-1}^2+\mbf F_{r-1}^2\mbf T_r=(v+v^{-1})
\mbf F_{r-1}\mbf T_r\mbf F_{r-1}.  \\
\end{align*}
\end{proposition}
\begin{proof}
The verification of the relations is essentially reduced to the finite
type computations, which is given in \cite[Proposition 6.2.1]{FL14}, and hence we omit it.
\end{proof}
The following lemma is an analogue of \cite[Corollary 4.6.6]{FL14}, which follows by a standard Vandermonde determinant type argument.
\begin{lemma}
The algebra $\mbf U_{n,d}^\mak d$ is generated by $\mbf E_i, \mbf F_i, \mbf K_i$ for $1\leq i <r$ and $\mbf J_\pm, \mbf T_0, \mbf T_r$.
\end{lemma}
\subsection{A raw comultiplication}
For any $A  \in \Xi_{n,d}$, we set
\begin{equation}
[A^\pm]=[(A, +)]+[(A,-)] \in \mbf S_{n,d}^\mak d.
\end{equation}
Denote by $\wit{\mbf S}_{n,d}^\mak d$ the subalgebra of $\mbf S_{n,d}^\mak d$ spanned by $[A^{\pm}]$ for all $A \in \Xi_{n,d}$.
Then $\mbf S_{n,d}^\mak d$ admits the following decomposition
\begin{align*}
\mbf S_{n,d}^\mak d =\mbf J_+\wit{\mbf S}_{n,d}^\mak d\oplus \mbf J_-\wit{\mbf S}_{n,d}^\mak d.
\end{align*}

For any $f \in \wit{\mbf S}_{n,d}^\mak d$ and $(L,L') \in \mcal X_{n,d}^\mak d\times \mcal X_{n,d}^\mak d$, we have $f(L,L')=f(g L,g L'), \forall g\in {\rm O}_F(V)$.
Moreover, the set $\{\{A^\pm \}_d\mid A \in \Xi_{n,d}\}$ forms a basis of $\wit{\mbf S}_{n,d}^\mak d$ (called the canonical basis), where $\{A^\pm \}_d=\{(A, +)\}_d+\{(A, +)\}_d$.

Let $V''$ be an isotropic $F$-subspace of dimension $d''$.
Then $V'=V''^\bot/V''$ is a vector space of dimension $2d'$ with its symmetric form induced from $V$ where $d'=d-d''$.

Recall $\mcal X_{n,d}$ and $\mbf S_{n,d}$ of affine type $\mrm A$ in \cite{FLLLW20, Lu99}.
The Schur algebra $\mbf S_{n,d}$ admits a standard basis $\{[A]^{\mbf a} \mid A \in \Theta_{n,d}\}$ and a canonical basis $\left\{\{A\}_d^{\mbf a}\mid  A \in \Theta_{n,d}\right\}$,
where
\begin{equation}
\Theta_{n,d}=\bigg\{A=(a_{ij})_{i,j \in \mbb Z} \in {\rm Mat}_{\mbb Z\times \mbb Z}(\mbb N) \ \Big| \ a_{ij}=a_{i+n,j+n}, \ \forall i,j \in \mbb Z, \ \sum_{1\leq i\leq n}a_{ij}=d\bigg\}.
\end{equation}
Moreover, the Lusztig subalgebra $\mbf U_{n,d}$ of $\mbf S_{n,d}$ is generated by the Chevalley generators $\mbf e_i, \mbf f_i$ and $\mbf h_i$ such that $\mbf e_{i+n}=\mbf e_i, \mbf f_{i+n}=\mbf f_i$ and $\mbf h_{i+n}=\mbf h_i$.

Given a periodic chain $L$ in $\mcal X_{n,d}^\mak d$,
we can define a periodic chain $L''=\pi''(L) \in \mcal X_{n,d''}$ (of affine type $\mrm A$) by setting $L_i''=L_i\cap V''$.
We also define a periodic chain $L'=\pi^{\natural}(L) \in \mcal{X}_{n,d'}^\mak d$ by setting $L_i'=(L_i\cap V''^\bot+V'')/V''$ for all $i$.
Given any pair $(L',L'')\in \mcal X_{n,d'}^\mak d \times \mcal X_{n,d''}$, we set
\begin{equation*}
\mcal Z_{L',L''}^\mak d=\{L\in \mcal X_{n,d}^\mak d \mid \pi^\natural(L)=L',\pi''(L)=L''\}.
\end{equation*}
The same argument as in \cite{FLLLW20} shows that there is a well-defined map
\begin{align}\label{def}
\wit{\Delta}^\mak d: \wit{\mbf S}_{n,d}^\mak d\rightarrow\wit{\mbf S}_{n,d'}^\mak d\otimes \mbf S_{n,d''}, \ \forall d'+d''=d,
\end{align}
given by
\begin{align}\label{defi}
\wit{\Delta}^\mak d(f)(L',\wit L',L'',\wit L'')=\sum_{\wit L \in \mcal Z_{\wit L',\wit L''}^\mak d}f(L,\wit L), \ \ \ \forall L',\wit L'\in \mcal X_{n,d'}^\mak d,L'',\wit L''\in \mcal X_{n,d''},
\end{align}
where $L$ is a fixed element in $ \mcal Z_{L',L''}^\mak d$.

Following the argument of \cite[Proposition 1.5]{Lu00}, which is formal and not reproduced here, we have the following proposition.
\begin{proposition}
The map $\wit{\Delta}^\mak d$ is an algebra homomorphism.
\end{proposition}

Now we determine how the map $\wit{\Delta}^\mak d$ acts on the generators.
\begin{proposition}\label{1000}
For any $i \in [1,r-1]$, we have
\begin{align*}
\wit{\Delta}^\mak d(\mbf E_i) & =\mbf E_i'\otimes \mbf h_{i+1}''\mbf h_{n-i}''^{-1}+\mbf H_{i+1}'^{-1}\otimes \mbf e_{i}''\mbf h_{n-i}''^{-1}+\mbf H_{i+1}'\otimes \mbf f_{n-i}''\mbf h_{i+1}'', \\
\wit{\Delta}^\mak d(\mbf F_i) & =\mbf F_i'\otimes \mbf h_i''^{-1}\mbf h_{n+1-i}''+\mbf H_i'\otimes \mbf f_i''\mbf h_{n+1-i}''+\mbf H_i'^{-1}\otimes \mbf e_{n-i}''\mbf h_i''^{-1},\\
\wit{\Delta}^\mak d(\mbf T_r) & =\mbf T_r'\otimes \mbf k_r''+v\mbf H_{r+1}'^{-1}\otimes \mbf e_r''\mbf h_r''^{-1}+v^{-1}\mbf H_{r+1}'\otimes \mbf f_r''\mbf h_{r+1}'',\\
\wit{\Delta}^\mak d(\mbf T_0) & =\mbf T_0'\otimes \mbf k_0''+v\mbf H_1'^{-1}\otimes \mbf e_0''\mbf h_0''^{-1}+v^{-1}\mbf H_1'\otimes \mbf f_0''\mbf h_1'',\\
\wit{\Delta}^\mak d(\mbf K_i) & =\mbf K_i'\otimes \mbf k_i''\mbf k_{n-i}''^{-1}.
\end{align*}
Here the subperscripts $'$ and $''$ indicate that the underlying Chevalley generators lie in $\wit{\mbf S}_{n,d'}^\mak d$ and $\mbf S_{n,d''}$, respectively.
\end{proposition}
\begin{proof}
The cases for $\mbf E_i,\mbf F_i$ and $\mbf K_i$ follow from a similar argument to \cite[Lemma 3.2.1]{FL15}.
For any $L \in \mcal X_{n,d}^\mak d$, we have
\begin{align*}
|L_{i+1}/L_i|=|L_{i+1}'/L_i'|+|L_{i+1}''/L_i''|+|L_{n-i}''/L_{n-1-i}''|.
\end{align*}
By definition, we have
\begin{align*}
\wit{\Delta}^\mak d(\mbf T_0)(L',\check{L}',L'',\check{L}'')
=v^{1-|\check{L}_{1}/\check{L}_0|}\sharp S,
\end{align*}
where $S=\{\check{L}\in \mcal Z_{\check{L}',\check{L}''}^\mak d\mid L_0\cap\check{L}_0\overset{1}{\subset}L_0, L_j=\check{L}_j, \ \forall 1\leq j\leq r\}$.
The set $S$ is nonempty only when the quadruple $(L',\check{L}',L'',\check{L}'')$ is in one of the following three cases.
\begin{align*}
  &({\rm i}) \   L_0'\cap \check{L}_0'\overset{1}{\subset}L_0', \ L_j'=\check{L}_j'  \  {\rm for }  \ 1\leq j\leq r, \ {\rm and}  \ L''=\check{L}'', \\
  &({\rm ii}) \  L'=\check{L}', \ {\rm and} \ L_0''\overset{1}{\subset}\check{L}_0'', \  L_j''=\check{L}_j'' \ {\rm for} \ j\neq 0\ {\rm mod}\ n,\\
  &({\rm iii}) \  L'=\check{L}', \ {\rm and} \ L_0''\overset{1}{\supset}\check{L}_0'', \ L_j''=\check{L}_j'' \ {\rm for} \ j\neq 0\ {\rm mod}\ n.
\end{align*}

We now compute the number $\sharp S$ in case $({\rm i})$.
This amounts to count all lines $<u>$, spanned by the vector $u$ such that $L_0+<u> \subset L_1$ and $\pi''(L_0+<u>)=L_0'', \pi^{\natural}(L_0+<u>)=L_0'+\check{L}_0'$.
We need to find those $u$ such that $u+V''=u'$,
where $u'$ is a fixed element in $V''^{\bot}/V''$ such that $L_0'+<u'>=L_0'+\check{L}_0'$.
Then $u$ is of the form $t+w$ where $w\in L_1''$ and $t$ is a fixed element in $V''^{\bot}$ such that $t+V''=u'$.
Since adding $w$ by any vector in $L_0''$ does not change the resulting space $L_0+<u>$, we see that the freedom of choice for $w$ is $L_1''$ mod $L_0''$,
i.e., $L_1''/L_0''$.
So the value of $\widetilde{\Delta}^\mak d(\mbf T_0)(L',\check{L}',L'',\check{L}'')$ is equal to
\begin{align*}
v^{1-|\check{L}_1/\check{L}_0|}q^{|L_1''/L_0''|}=v^{1-|\check{L}_1'/\check{L}_0'|}
v^{-|\check{L}_0''/\check{L}_{-1}''|+|\check{L}_1''/\check{L}_0''|}=(\mbf T_0'\otimes \mbf k_0'')(L',\check{L}',L'',\check{L}'').
\end{align*}

For case $({\rm ii})$, $S$ consists of only one element.
So we see that the value of $\wit{\Delta}^\mak d(\mbf T_0)(L',\check{L}',L'',\check{L}'')$ is equal to
\begin{align*}
v^{1-|\check{L}_1/\check{L}_0|}=v^{1-|\check{L}_1'/\check{L}_0|}v^{-|\check{L}_1''/
\check{L}_0''|}v^{-|\check{L}_0''/\check{L}_{-1}''|}=(v\mbf H_1'^{-1}\otimes
\mbf e_0''\mbf h_0''^{-1})(L',\check{L}',L'',\check{L}'').
\end{align*}

For case $({\rm iii})$, we amount to count the lattice $\mcal L$ such that $|L_0/\mcal L|=1$ and $\mcal L\cap V''=\check{L}_0''$.
Also, we have that the number of $\mcal L$ is equal $\sharp S$, then
\begin{align*}
  \sharp S&=\sharp\{\mcal L \mid \check{L}_0''+L_{-1}\subset \mcal L\overset{1}{\subset}L_0\}-\sharp\{\mcal L \mid L_0''+L_{-1}\subset
  \mcal L\overset{1}{\subset}L_0\}\\
   &=q^{|L_0'/L_{-1}'|+|L_1''/L_0''|}=q^{|\check{L}_0'/\check{L}_{-1}'|+
   |\check{L}_1''/\check{L}_0''|-1}.
\end{align*}
 Thus, the value of $\wit{\Delta}^\mak d(\mbf T_0)(L',\check{L}',L'',\check{L}'')$ is equal to
\begin{align*}
  v^{1-|\check{L}_1/\check{L}_0|}q^{|\check{L}_0'/\check{L}_{-1}'|
  +|\check{L}_1''/\check{L}_0''|-1}&=v^{|\check{L}_1'/\check{L}_0'|-1}v^{-|\check{L}_0''/
  \check{L}_{-1}''|+|\check{L}_1''/\check{L}_0''|} \\
   & =(v^{-1}\mbf H_1'\otimes \mbf f_0''\mbf h_1'')(L',\check{L}',L'',\check{L}'').
\end{align*}
We check the identity of $\mbf T_0$, the case of $\mbf T_r$ is similar,
and we omit it.
\end{proof}
Recall $\Lambda_{n,d}$ in the affine type $\mrm A$ from \cite{FLLLW20}
\begin{align*}
\Lambda_{n,d}=\bigg\{\lambda=(\lambda_i)_{i \in \mbb Z} \in \mbb N^{\mbb Z} \ \Big| \ \lambda_i=\lambda_{i+n}, \ \forall i \in \mbb Z, \ \sum_{1\leq i\leq n}\lambda_i=d \bigg\}.
\end{align*}
Similarly, the set $\mcal X_{n,d}$ can be decomposed as follows:
\begin{align*}
\mcal{X}_{n,d}=\bigsqcup_{\mbf a=(a_i)\in \Lambda_{n,d}}\mcal{X}_{n,d}(\mbf a), \ \  \text{where} \ \mcal{X}_{n,d}(\mbf a)=\{L\in \mcal X_{n,d} \mid |L_i/L_{i-1}|=a_i, \forall i\in \mbb Z\}.
\end{align*}

Given $\mbf a,\mbf b \in \Lambda_{n,d}^{\mak d}$,
let $\wit{\mbf S}_{n,d}^{\mak d}(\mbf b,\mbf a)$ be the subspace of $\wit{\mbf S}_{n,d}^{\mak d}$ spanned by the elements $[A^\pm]$ such that ${\rm ro}(\mak a)=\mbf b$ and ${\rm co}(\mak a)=\mbf a$ for all $A \in \Xi_{n,d}$. Similarly, for $\mbf a,\mbf b \in \Lambda_{n,d}$,
we define the affine type $\mrm A$ counterpart $\mbf S_{n,d}(\mbf b,\mbf a)$.
Let $\wit{\Delta}_{\mbf b',\mbf a',\mbf b'',\mbf a''}^{\mak d}$
be the component of $\wit{\Delta}^{\mak d}$ from $\wit{\mbf S}_{n,d}^{\mak d}(\mbf b,\mbf a)$ to $\wit{\mbf S}_{n,d'}^{\mak d}(\mbf b',\mbf a')\otimes \mbf S_{n,d''}(\mbf b'',\mbf a'')$ such that $b_i=b_i'+b_i''+b_{1-i}''$ and $ a_i=a_i'+a_i''+a_{1-i}''$, for $i \in \mbb Z$.
We set
\begin{align*}
s(\mbf b',\mbf a',\mbf b'',\mbf a'')=\sum_{1\leq k\leq j\leq n}b_k'b_j''-a_k'a_j'',
\end{align*}
and
\begin{align*}
u(\mbf b'',\mbf a'')=\frac{1}{2}\bigg(\sum_{\overset{1\leq k,j\leq n}{k+j\leq n+1}}b_k''b_j''-a_k''a_j''+\sum_{r< k \leq n}a_k''-b_k''\bigg),
\end{align*}
for all $\mbf b',\mbf a' \in \Lambda_{n,d'}^{\mak d}$ and $\mbf b'',\mbf a'' \in \Lambda_{n,d''}$.
We renormalize $\wit{\Delta}^{\mak d}$ to be $\Delta^{\mak d^\dagger}$ by letting
\begin{align*}
\Delta_{\mbf b',\mbf a',\mbf b'',\mbf a''}^{\mak d^\dagger} & =v^{s(\mbf b',\mbf a',\mbf b'',\mbf a'')+u(\mbf b'',\mbf a'')}
\wit{\Delta}_{\mbf b',\mbf a',\mbf b'',\mbf a''}^\mak d
  \\
\Delta^{\mak d^\dagger} &= \bigoplus_{\mbf b',\mbf a',\mbf b'',\mbf a''}
\Delta_{\mbf b',\mbf a',\mbf b'',\mbf a''}^{\mak d^\dagger}: \wit{\mbf S}_{n,d}^\mak d\rightarrow\wit{\mbf S}_{n,d'}^
\mak d\otimes\mbf S_{n,d}.
\end{align*}

Now let us study the restriction of $\Delta^{\mak d^\dagger}$ to $\mbf U_{n,d}^\mak d$.
\begin{proposition}
Let $d=d'+d''$.
For all $i \in [1,r-1]$, we have
\begin{align*}
   \Delta^{\mak d^\dagger}(\mbf E_i) &=\mbf E_i'\otimes \mbf k_i''+1\otimes \mbf e_i''+\mbf K_i'\otimes \mbf f_{n-i}''\mbf k_i'', \\
   \Delta^{\mak d^\dagger}(\mbf F_i) &=\mbf F_i'\otimes \mbf k_{n-i}''+\mbf K_i'^{-1}\otimes \mbf k_{n-i}''\mbf f_i''+1\otimes \mbf e_{n-i}'' ,\\
   \Delta^{\mak d^\dagger}(\mbf T_0) &=\mbf T_0'\otimes \mbf k_0''+v^{-2d'-d''+1}\otimes \mbf e_0''+v^{-2+2d'+d''}\otimes \mbf f_0''\mbf k_0'',\\
   \Delta^{\mak d^\dagger}(\mbf T_r) &=\mbf T_r'\otimes \mbf k_r''+1\otimes \mbf e_r''+v^{-1}\otimes\mbf f_r''\mbf k_r'',\\
   \Delta^{\mak d^\dagger}(\mbf K_i) &=\mbf K_i'\otimes \mbf k_i''\mbf k_{n-i}''^{-1}.
\end{align*}
Here the superscripts follows the same convention as in Proposition \ref{1000}.
\end{proposition}

Recall $\xi_{d,i,c}:\mbf S_{n,d}\rightarrow\mbf S_{n,d}$ in affine type $\mrm A$ from \cite{FL15}.
We define the algebra homomorphism
\begin{equation}\label{1212}
  \Delta^{\mak d}=(1\otimes \xi_{d'',0,-(2d'+d''-1)})\circ\Delta^{\mak d^\dagger}: \wit{\mbf S}_{n,d}^\mak d\rightarrow\wit{\mbf S}_{n,d'}^
  \mak d\otimes\mbf S_{n,d''}.
\end{equation}

Let $\mbf a, \mbf b \in \Lambda_{n,d}^\mak d$.
Fix $L \in \mcal X_{n,d}^\mak d(\mbf b)$ and let $P_{\mbf b}=\text{Stab}_{{\rm O}_F(V)}(L)$.
We have a natural embedding
\begin{align*}
\iota_{\mbf b, \mbf a}: \mcal X_{n,d}^\mak d(\mbf a)\rightarrow\mcal X_{n,d}^\mak d
(\mbf b)\times \mcal X_{n,d}^\mak d(\mbf a), \ \ \  L'\mapsto (L,L').
\end{align*}
It is well known that $\iota_{\mbf b,\mbf a}$ induces the following isomorphism of $\mcal A$-modules:
\begin{align*}
\iota_{\mbf b, \mbf a}^\ast:\mcal A_{{\rm O}_F(V)}(\mcal X_{n,d}^\mak d(\mbf b)\times \mcal X_{n,d}^\mak d(\mbf a))\rightarrow\mcal A_{P_{\mbf b}}
(\mcal X_{n,d}^\mak d(\mbf a)).
\end{align*}
Let
\begin{align*}
\mcal X_{\mbf a, \mbf a', \mbf a''}^\mak d=\{L\in \mcal X_{n,d}^\mak d(\mbf a) \mid \pi^\natural(L)\in \mcal X_{n,d'}^\mak d(\mbf a'),\pi''(L)\in \mcal X_{n,d''}(\mbf a'')\}.
\end{align*}
Then we have the following diagram
\begin{align*}
\mcal X_{n,d}^\mak d(\mbf a)\overset{\iota}{\longleftarrow}
\mcal X_{\mbf a, \mbf a', \mbf a''}^\mak d\overset{\pi}
{\longrightarrow}\mcal X_{n,d'}^\mak d(\mbf a')\times\mcal X_{n,d''}
(\mbf a''),
\end{align*}
where $\iota$ is the imbedding and  $\pi(L)=(\pi^\natural(L),\pi''(L))$.
By identifying $\mcal A_{P_{\mbf b'}\times P_{\mbf b''}}(\mcal X_{n,d}^\mak d(\mbf a')\times \mcal X_{n,d''}(\mbf a''))=\mcal A_{P_{\mbf b'}}
(\mcal X_{n,d'}^\mak d(\mbf a'))\times \mcal A_{P_{\mbf b''}}(\mcal X_{n,d''}(\mbf a''))$,
we have the following map
\begin{align*}
\pi_!\iota^\ast: \mcal A_{P_{\mbf b}}(\mcal X_{n,d}^\mak d(\mbf a))\longrightarrow \mcal A_{P_{\mbf b'}}(\mcal X_{n,d'}^\mak d(\mbf a'))\times \mcal A_{P_{\mbf b''}}(\mcal X_{n,d''}(\mbf a'')).
\end{align*}
By a similar argument as for \cite[Lemma 1.3.5]{FL15}, the following diagram commutes:
   \begin{equation}
  \label{comm-coproduct}
  \xymatrixrowsep{.5in}
\xymatrixcolsep{.8in}
\xymatrix{
\mcal A_{{\rm O}_F(V)}(\mcal X_{n,d}^\mak d(\mbf b)\times\mcal X_{n,d}^\mak d(\mbf a)) \ar[d]_{}
\ar@{->}[r]^-{\iota^*_{\mbf b, \mbf a}}
\ar@{->}[d]_{\widetilde \Delta^\mak d_{\mbf b', \mbf a', \mbf b'', \mbf a''}}
&
\mcal A_{P_{\mbf b}}(\mcal X_{n,d}^\mak d(\mbf a))
\ar@{->}[d]^{ \pi_! \iota^*}
\\
\smxylabel{
\substack{
\mcal A_{{\rm O}_F(2d')}(\mcal X_{n,d'}^\mak d(\mbf b')\times \mcal X_{n,d'}^\mak d(\mbf a') )  \\
\otimes \\
   \mcal A_{{\rm O}_F(2d'')}(\mcal X_{n,d''}(\mbf b'')\times \mcal X_{n,d'}(\mbf a'') )
}
}
\ar@{->}[r]^-{ \iota^*_{\mbf b', \mbf a'} \otimes \iota^*_{\mbf b'', \mbf a''}}
&
\smxylabel{
\mcal A_{P_{\mbf b'}}(\mcal X_{n,d'}^\mak d(\mbf a'))\otimes \mcal A_{P_{\mbf b''}}(\mcal X_{n,d''}(\mbf a''))
}
}
\end{equation}

Recall that $\Delta^\mak d:\wit{\mbf S}_{n,d}^\mak d\rightarrow
\wit{\mbf S}_{n,d'}^\mak d\otimes\mbf S_{n,d''}$
from (\ref{1212}).
By an argument similar to \cite[Section 2.4]{FL15} and (\ref{comm-coproduct}), the positivity for
$\Delta^\mak d$ follows from \cite[Theorem 8]{Br03}.
\begin{proposition}\label{chen22}
For $A \in \Xi_{n,d}$, write
\begin{align*}
\Delta^\mak d(\{A^\pm\}_d)=\sum_{A'\in \Xi_{n,d'},A''\in \Theta_{n,d''}}h_A^{A',A''}\{A'^\pm\}_{d'}^\mak d
\otimes\{A''\}_{d''}^\mbf a.
\end{align*}
Then $h_{A}^{A',A''} \in \mbb N[v,v^{-1}]$ for all $A,A'$ and $A''$.
\end{proposition}

Let $\wit{\mbf U}_{n,d}^\mak d$ be the subalgebra of $\wit{\mbf S}_{n,d}$ generated by all elements $[B^\pm]$ such that $B$ or $B-E_\theta^{h,h+1}$ is diagonal for various $h \in \mbb Z$.
It is clear that $\wit{\mbf U}_{n,d}^\mak d$ is generated by $\mbf E_i,\mbf F_i, \mbf K_i$ and $\mbf T_0, \mbf T_r$ for $i \in [1,r-1]$.

\begin{proposition}\label{1216}
Let $d=d'+d''$.
 We have a homomorphism $\Delta^\mak d:\widetilde{\mbf U}_{n,d}^\mak d\rightarrow\widetilde{\mbf U}_{n,d'}^\mak d\otimes \mbf U_{n,d''}$.
 More precisely, for all $i \in [1,r-1]$, we have
\begin{align*}
  \Delta^\mak d(\mbf E_i) &=\mbf E_i'\otimes \mbf k_i''+1\otimes \mbf e_i''+\mbf K_i'\otimes \mbf f_{n-i}''\mbf k_i'',  \\
  \Delta^\mak d(\mbf F_i) &=\mbf F_i'\otimes \mbf k_{n-i}''+\mbf K_i'^{-1}\otimes \mbf k_{n-i}''\mbf f_i''+1\otimes \mbf e_{n-i}'', \\
  \Delta^\mak d(\mbf T_0)&=\mbf T_0'\otimes\mbf k_0''+1\otimes \mbf e_0''+v^{-1}\otimes \mbf f_0\mbf k_0,\\
  \Delta^\mak d(\mbf T_r)&=\mbf T_r'\otimes\mbf k_r''+1\otimes \mbf e_r''+v^{-1}\otimes \mbf f_r\mbf k_r,\\
  \Delta^\mak d(\mbf K_i) &=\mbf K_i'\otimes \mbf k_i''\mbf k_{n-i}''^{-1}.
\end{align*}
\end{proposition}

Recall the comultiplication $\Delta$ in the affine type $\mrm A$ from \cite{FL15}. This is an algebra homomorphism
\begin{align*}
\Delta:\mbf S_{n,d}\rightarrow\mbf S_{n,d'}\otimes \mbf S_{n,d''},
\end{align*}
defined by
\begin{equation}\label{chen1}
\begin{split}
\Delta(\mbf e_i)&=\mbf e_i'\otimes \mbf k_i''+1\otimes \mbf e_i'',\\
\Delta(\mbf f_i)&=\mbf f_i'\otimes 1+\mbf k_i'^{-1}\otimes \mbf f_i'',\\
\Delta(\mbf k_i)&=\mbf k_i'\otimes \mbf k_i'',  \ \ \ \  \forall i\in [1,n].
\end{split}
\end{equation}
Here the superscripts follows the same convention as in Proposition \ref{1000}.
\begin{proposition}\label{1217}
The following coassociativity holds on $\widetilde{\mbf U}_{n,d}^\mak d$:
\begin{align*}
(1\otimes \Delta)\Delta^\mak d=(\Delta^\mak d\otimes 1)\Delta^\mak d.
\end{align*}
\end{proposition}
\begin{proof}
we can compute directly on the generators.
\end{proof}
Now setting $d'=0$, we have $\mbf E_i'=0,\mbf F_i'=0,\mbf T_0'=0,\mbf T_r'=0$ and $\mbf K_i'=1$ in $\widetilde{\mbf S}_{n,0}^\mak d$, and $\Delta^\mak d$ becomes the following algebra homomorphism
\begin{equation}\label{chen2}
\begin{split}
&\jmath_{n,d}:\widetilde{\mbf S}_{n,d}^\mak d\rightarrow\mbf S_{n,d}\\
&\jmath_{n,d}(\mbf E_i)=\mbf e_i+\mbf k_i\mbf f_{n-i},  \ \ \ \ \ \ \ \ \ \
\jmath_{n,d}(\mbf F_i)=\mbf e_{n-i}+\mbf f_i\mbf k_{n-i},\\
&\jmath_{n,d}(\mbf T_0)=\mbf e_0+v^{-1}\mbf f_0\mbf k_0, \ \ \ \ \  \ \
\jmath_{n,d}(\mbf T_r)=\mbf e_r+v^{-1}\mbf f_r\mbf K_r,\\
&\jmath_{n,d}(\mbf K_i)=\mbf k_i\mbf k_{n-i}^{-1},  \ \ \  \forall i\in [1,r-1].
\end{split}
\end{equation}
It follows by restriction that we have also a homomorphism $\jmath_{n,d}: \widetilde{\mbf U}_{n,d}^\mak d\rightarrow \mbf U_{n,d}$.
The same argument as in \cite[Proposition 5.3.6]{FLLLW20} gives us the following proposition.
\begin{proposition}\label{1401}
The homomorphism $\jmath_{n,d}: \widetilde{\mbf S}_{n,d}^\mak d\rightarrow \mbf S_{n,d}$ (and $\jmath_{n,d}: \widetilde{\mbf U}_{n,d}^\mak d\rightarrow \mbf U_{n,d}$ ) is injective.
\end{proposition}

Proposition \ref{chen22} in our setting of $d'=0$ gives us the following proposition.

\begin{proposition}\label{sum canonical}
The map $\jmath_{n,d}$ sends a canonical basis element in $\widetilde{\mbf S}_{n,d}^\mak d$
 to a sum of canonical basis elements of $\mbf S_{n,d}$ with coefficients in $\mbb N[v,v^{-1}]$.
\end{proposition}
\subsection{Canonical basis of $\mbf U_{n,d}^\mak d$}
Let $\mak N=[\mak b_1]\ast\cdots\ast[\mak b_m]$ be a non-zero aperiodic monomial in $\mbf S_{n,d}^\mak d$, where $\mak b_i=(B_i,\alpha_i)$.
Without loss of generality, we set that $\alpha_1=+$, then $\mak N=\mbf J_{+}[B_1^\pm]\ast\cdots\ast[B_m^\pm]$.

By a similar argument as for \cite[Proposition 5.4.2]{FLLLW20},
we have the following proposition.
\begin{proposition}\label{chen1213}
Let $\mak N$ be an aperiodic monomial in $\mbf S_{n,d}^\mak d$.
Suppose that $\mak N=\sum c_\mak a\{\mak{a}\}_d$ where $c_\mak a \in \mbb Z[v,v^{-1}]$.
If $c_\mak a\neq 0$, then $\mak a$ must be aperiodic.
\end{proposition}
From Lemma \ref{chen121212} and Proposition \ref{chen1213}, we get the following theorem.
\begin{theorem}\label{1222}
The set $\{\{\mak a\}_d \mid \mak a \in \Xi_\mak d^{ap}\}$ forms a basis of $\mbf U_{n,d}^\mak d$, called the canonical basis.
Also, the set $\{\zeta_\mak a \mid \mak a \in \Xi_\mak d^{ap}\}$ forms a basis (called a monomial basis) of $\mbf U_{n,d}^\mak d$.
\end{theorem}
\begin{proof}
For $\mak a \in \Xi_\mak d^{ap}$, we have $\zeta_\mak a=[\mak a ]+{\rm lower \ terms}$ by Lemma \ref{chen121212}.
By Proposition \ref{chen1213}, we can write
$\zeta_\mak a \in \{\mak a \}_d+\sum_{\mak a'<\mak a}\mcal A\{\mak a'\}_d$,
where $\mak a' \in \Xi_\mak d^{ap}$.
 By an induction on $\mak a$ by the partial ordering, we conclude that $\{\mak a\}_d \in \mbf U_{n,d}^\mak d$.
 Since the set $\{\{\mak a\}_d \mid \mak a \in \Xi_\mak d^{ap}\}$ is clearly linearly independent and it forms a spanning set of $\mbf U_{n,d}^\mak d$ by Proposition \ref{chen1213}, it is a basis of $\mbf U_{n,d}^\mak d$.

Since the transition matrix from $\{\zeta_\mak a \mid \mak a \in \Xi_\mak d^{ap}\}$ to the canonical basis is uni-triangular, $\{\zeta_\mak a \mid \mak a \in \Xi_\mak d^{ap}\}$ forms a basis as well.
\end{proof}

We have the following corollary of Proposition \ref{sum canonical} and Theorem \ref{1222}.
\begin{corollary}
The set $\{\{A\}_d^{\mak d}\mid A \in \Xi_{n,d}^{ap}\}$ forms a basis (called a canonical basis) of $\widetilde{\mbf U}_{n,d}^\mak d$.
Moreover, the image of $\jmath_{n,d}$ of a canonical basis element in $\wit{\mbf U}_{
n,d}^\mak d$ is a sum of canonical
basis elements of $\mbf U_{n,d}$ with coefficients in $\mbb N [v, v^{-1}]$.
\end{corollary}

\section{Corresponding Quantum symmetric pair}
In this section, we introduce the transfer maps $\phi_{d,d-n}^\mak d$ on algebra $\mbf U_{n,d}^\mak d$, and construct (idempotented) quantum algebras $\widetilde{\mbf U}_n^\mak d$ and $\mbf U_n^\mak d$ from the projective systems of algebras $\{(\mbf U_{n,d}^\mak d, \phi_{n,d}^\mak d)\}_{d \in \mbb N}$.
We show that $\wit{\mbf U}_n^\mak d$ is a coideal subalgebra of $\mbf U(\widehat{\mak {sl}}_n)$, and $(\mbf U(\widehat{\mak {sl}}_n), \widetilde{\mbf U}_n^ \mak d)$ forms an affine quantum symmetric pair.

\subsection{The algebras $\mbf U_n^\mak d$ and $\dot {\mbf U}_n^\mak d$}
Recall \cite{Lu00} there exists a homomorphism $\chi_n: \mbf S_{n,n}\rightarrow \mbb Q(v)$ such that
\begin{align*}
\chi_n(\mbf e_i)=\chi_n(\mbf f_i)=0, \  \chi_n(\mbf h_i)=v.
\end{align*}
Following Lusztig \cite{Lu00}, we introduce the transfer map
\begin{align*}
\widetilde{\phi}_{d,d-n}^\mak d: \widetilde{\mbf S}_{n,d}^\mak d \rightarrow \widetilde{\mbf S}_{n,d-n}^\mak d,
\end{align*}
which is by definition the composition of the following homomorphisms $({\rm for} \  d\geq n)$
\begin{align}\label{chen121}
 \widetilde{\phi}_{d,d-n}^\mak d: \widetilde{\mbf S}_{n,d}^\mak d\overset{\widetilde{\Delta}^\mak d}{\longrightarrow}\widetilde{\mbf S}_{n,d-n}^\mak d\otimes \mbf S_{n,n}\overset{1\otimes \chi_n}{\longrightarrow}\widetilde{\mbf S}_{n,d-n}^\mak d.
\end{align}

\begin{proposition}
For $i \in [1,r-1]$, we have $\widetilde{\phi}_{d,d-n}^\mak d(\mbf E_i)=\mbf E_i', \widetilde{\phi}_{d,d-n}^\mak d(\mbf F_i)=\mbf F_i',\widetilde{\phi}_{d,d-n}^\mak d(\mbf K_i)=\mbf K_i'$ and $\widetilde{\phi}_{d,d-n}^\mak d(\mbf T_0)=\mbf T_0',\widetilde{\phi}_{d,d-n}^\mak d(\mbf T_r)=\mbf T_r'$.
\end{proposition}
Recall that $\mbf S_{n,d}^\mak d=\mbf J_+\wit{\mbf S}_{n,d}^\mak d\oplus \mbf J_-\wit{\mbf S}_{n,d}^\mak d$, we define the algebra homomorphism
\begin{align}\label{chen1214}
\phi_{n,d}^\mak d: \mbf S_{n,d}^\mak d\rightarrow\mbf S_{n,d-n}^\mak d, \ \   \phi_{n,d}^\mak d(\mbf J_{\pm}x)\mapsto \mbf J'_{\pm}\widetilde{\phi}_{n,d}^\mak d(x), \  \forall x \in \widetilde{\mbf S}_{n,d}^\mak d.
\end{align}
Now we consider the projective system $\{(\mbf U_{n,d}^\mak d,\phi_{d+n,d}^\mak d)\}_{d\in \mbb N}$ and its projective limit:
\begin{align*}
\mbf U_{n,\infty}^\mak d:=\underset{d}{\underset{\longleftarrow}{\rm lim}} \mbf U_{n,d}^\mak d=\Big\{x\equiv (x_d)_{d\in \mbb N}\in \prod_{d \in \mbb N}\mbf U_{n,d}^\mak d \
\big| \ \phi_{d,d-n}^\mak d(x_d)=x_{d-n}, \forall d\Big\}.
 \end{align*}
 Denote by $\phi_d^\mak d:\mbf U_{n,\infty}^\mak d\rightarrow \mbf U_{n,d}^\mak d$ the natural projection.
 Since the bar involution commutes with the transfer map $\phi_{d,d-n}^\mak d$,
 there exists a bar involution $\bar{\quad}:\mbf U_{n,\infty}^\mak d \rightarrow \mbf U_{n,\infty}^\mak d$.
 Also, we have an integral version: $\mbf U_{n,\infty;\mcal A}^\mak d=
 \underset{d}{\underset{\longleftarrow}{\rm lim}}\mbf U_{n,d;\mcal A}^\mak d$.
Since $\mbb Q(v)\otimes_\mcal A \mbf U_{n,d;\mcal A}^\mak d
=\mbf U_{n,d}^\mak d$ for all $d$, we have $\mbb Q(v)\otimes_\mcal A\mbf U_{n,\infty;\mcal A}^\mak d =\mbf U_{n,\infty}^\mak d$.
Similarly, we get the projective limit $\wit{\mbf U}_{n,\infty}^\mak d=\underset{d}{\underset{\longleftarrow}{\rm lim}}\wit{\mbf U}_{n,d}^\mak d$ of the projective system
$\{(\wit{\mbf U}_{n,d}^\mak d,\wit{\phi}_{d+n,d}^\mak d)\}_{d\in \mbb N}$ and denote by $\wit{\phi}_d^\mak d: \wit{\mbf U}_{n,\infty}^\mak d\rightarrow \wit{\mbf U}_{n,d}^\mak d$ the natural projection.

We define elements $\mbf E_i, \mbf F_i,\mbf K_i^{\pm1}$ and $\mbf T_0, \mbf T_r, \mbf J_{\pm}$ for $i \in [1,r-1]$ in $\mbf U_{n,\infty}^\mak d$ by
\begin{align*}
  (\mbf E_i)_d&=\mbf E_{i,d}, \ (\mbf F_i)_d=\mbf F_{i,d}, \ (\mbf K_i^{\pm1})_d=\mbf K_{i,d}^{\pm1}, \\
  (\mbf T_0)_d&=\mbf T_{0,d}, \ (\mbf T_r)_d=\mbf T_{r,d}, \ (\mbf J_{\pm})_d=\mbf J_{\pm,d},
\end{align*}
where the $d \in \mbb N$ in the subscript of $\mbf E_{i,d}$ etc. indicates $\mbf E_{i,d}$ is a copy of the Chevalley generator $\mbf E_i$ in $\mbf U_{n,d}^\mak d$.
Let $\widetilde{\mbf U}_n^\mak d$ be the subalgebra of $\mbf U_{n,\infty}^\mak d$ generated by (the Chevalley generators) $\mbf E_i,\mbf F_i,\mbf K_i^{\pm1}$ and $\mbf T_0, \mbf T_r$ for $i\in [1,r-1]$.
Also, denote $\mbf U_n^\mak d$ to be the subalgebra of $\mbf U_{n,\infty}^\mak d$ generated by $\wit{\mbf U}_n^\mak d$ and $\mbf J_\pm$.

Let
\begin{align*}
 \mbb Z_n^\mak d=\{\lambda=(\lambda_i)_{i\in \mbb Z} \mid \lambda_i \in \mbb Z, \lambda_i=\lambda_{i+n},\lambda_i=\lambda_{1-i}\}.
\end{align*}
Let $|\lambda|=\lambda_1+\cdots+\lambda_n$.
Define an equivalence relation $\approx$ on $\mbb Z_n^\mak d$ by letting $\lambda \approx \mu$ if and only if $\lambda-\mu=(\cdots,p,p,p,\cdots)$, for some even integer $p$.
Let $\mbb Z_n^\mak d/\approx$ be the set of equivalence classes with respect to the equivalence classes with respect to the equivalence relation $\approx$;
and let $\widehat{\lambda}$ be the equivalence class of $\lambda$.

Fix $\widehat{\lambda}\in \mbb Z_n^\mak d/\approx$, we define the element $1_{\widehat{\lambda}} \in  \mbf U_{n,\infty}^\mak d$ as follows. $(1_{\widehat{\lambda}})_d=0$ if $d \not\equiv |\lambda| \ ({\rm mod} \ 2n)$.
If $d=|\lambda|+pn$ for some even integer $p$, we have $(1_{\widehat{\lambda}})_d=1_{\lambda+pI}$, where $1_{\lambda}=[D_\lambda^\pm]$ and $D_{\lambda}$ is diagonal matrix whose diagonal is $\lambda$.
Here $\lambda+pI$ is understood as $\lambda+(\cdots,p,p,p,\cdots)$,
and $1_{\lambda+pI} \in \mbf U_{n,d}^\mak d$ is understood to be zero if there is a negative entry in $\lambda+pI$.

\begin{definition}
Let $\dot{\mbf U}_n^\mak d$ be the $\mbf U_n^\mak d$-bimodule in $\mbf U_{n,\infty}^\mak d$ generated by $1_{\widehat{\lambda}}$ for all $\widehat{\lambda}\in \mbb Z_n^\mak d/\approx$.
\end{definition}

Similarly, we have $\mbb Q(v)\otimes_\mcal A\dot{\mbf U}_{n;\mcal A}^\mak d
=\dot{\mbf U}_n^\mak d$, where $\dot{\mbf U}_{n;\mcal A}^\mak d$ is the $\mcal A$-subalgebra of $\mbf U_{n,\infty}^\mak d$ such that generated by $\mbf E_i^{(a)}1_{\widehat{\lambda}}, \mbf F_i^{(a)}1_{\widehat{\lambda}}$ and $\mbf T_0^{(a)}1_{\widehat{\lambda}}, \mbf T_r^{(a)}1_{\widehat{\lambda}}, \mbf J_{\pm}1_{\widehat{\lambda}}$, for $i\in [1,r-1]$ and $a \in \mbb N$.

We set
\begin{align*}
&\widetilde{\Xi}_n=\{A=(a_{ij}) \in {\rm Mat}_{\mbb Z\times\mbb Z}(\mbb Z)\mid a_{ij}\geq 0 \ (i\neq j), a_{ij}=a_{1-i,1-j}=a_{i+n,j+n}, \forall i,j\in \mbb Z\}, \\
&\widetilde{\Xi}_{n,\mak d}=\widetilde{\Xi}_n \times \{+\}\sqcup \widetilde{\Xi}_n\times \{-\},\\
&\widetilde{\Xi}_{n,\mak d}^{ap}=\{(A,\alpha)\in \widetilde{\Xi}_{n,\mak d}\mid (A, \alpha) \ {\rm is \ aperiodic}\}.
\end{align*}
For $\mak a=(A, \alpha) \in \widetilde{\Xi}_{n,\mak d}$, we shall denote by
$$|\mak a|=d,$$
if $\sum_{i=i_0+1}^{i_0+n}\sum_{j\in \mbb Z}a_{ij}=2d$ for some (or each) $i_0 \in \mbb Z$.
 We set, for $d \in \mbb Z$,
\begin{align*}
\widetilde{\Xi}_{n,\mak d}^d=\{\mak a \in \widetilde{\Xi}_{n,\mak d}\mid |\mak a|=d\},  \ \   \widetilde{\Xi}_{n,\mak d}=\sqcup_d \widetilde{\Xi}_{n,\mak d}^d.
\end{align*}
Clearly, we have that $\Xi_\mak d\subset \widetilde{\Xi}_{n,\mak d}^d$.

We define an equivalence relation $\approx$ on $\widetilde{\Xi}_{n,\mak d}^{ap}$ by
\begin{align*}
\mak a=(A,\alpha)\approx\mak b=(B,\alpha') \  {\rm iff} \ \alpha=\alpha' \ {\rm and}\ A-B=pI, {\rm for \ some \ even \ integer} \ p,
\end{align*}
where $I=\sum_{1\leq i\leq r}E_{\theta}^{ii}$,
and let $\widehat{\mak a}$ be the equivalence class of $\mak a$.
We define
${\rm ro}(\widehat{\mak a})=\widehat{{\rm ro}(\mak a)} \ {\rm and}\  {\rm co}(\widehat{\mak a})=\widehat{{\rm co}(\mak a)}$,
and they are elements in $\mbb Z_n^\mak d/\approx$.
We can define the element $\zeta_{\widehat {\mak a}}$ in $\dot{\mbf U}_n^\mak d$ by $(\zeta_{\widehat{\mak a}})_d=0$ unless $d=|\mak a | \ {\rm mod}\ 2n$, and if $|\mak a|=d+p/2n$ for some even integer $p$, $(\zeta_{\widehat{\mak a}})_d=\zeta_{\mak a+pI}$, where $\zeta_{\mak a+pI}$ is the monomial basis attached to $\mak a+pI$ in Theorem \ref{1222}.
Since $\phi_{d,d-n}^\mak d(\zeta_{\mak a+pI})=\zeta_{\mak a+(p-2)I}$, we see that $\zeta_{\widehat{\mak a}} \in \dot{\mbf U}_n^\mak d$.

The following linear independence is reduced to the counterpart at the Schur algebra level,
by an argument similar to \cite[Theorem 5.5]{LW15}.
\begin{proposition}
The set $\{\zeta_{\widehat{\mak a}} \mid \widehat{\mak a} \in \widetilde{\Xi}_{n,\mak d}^{ap}/\approx\}$ is linearly independent.
\end{proposition}
Next, we show that $\{\zeta_{\widehat{\mak a}}\mid \widehat{\mak a} \in \widetilde{\Xi}_{n,\mak d}^{ap}/\approx\}$ is indeed a basis for $\dot{\mbf U}_n^\mak d$.
For simplicity, we write $\mbf F_r$ (resp. $ \mbf F_n$) for $\mbf T_r $ (resp. $\mbf T_0$) and $\mbf F_{n-i}$ for $\mbf E_i$ for $i \in [1,r-1]$.
For $\lambda \in \Lambda_{n,d}^\mak d$ and a pair $(\mbf i,\mbf a)$, where $\mbf i=(i_1,\cdots, i_s)$ and $\mbf a=(a_1,\cdots,a_s)$ with $1\leq i_j\leq n$ and $a_j \in \mbb N$ for all $j$,
we denote
\begin{align*}
_{d}\mcal D_{\mbf i,\mbf a,\lambda}=
\mbf F_{i_1}^{(a_1)}\cdots\mbf F_{i_s}^{(a_s)}1_{\lambda} \in \mbf U_{n,d}^\mak d,
\end{align*}
Then $\mbf J_{\pm}{_{d}\mcal D_{\mbf i,\mbf a,\lambda}}$ exhaust all possible monomials in $\mbf U_{n,d}^\mak d$.

The same argument as for \cite[Proposition 6.2.3]{FLLLW20} gives us the following proposition.
\begin{proposition}
Fix a triple $(\mbf i,\mbf a,\lambda)$ with $|\lambda|=d$.
There is a finite subset $\mcal I_{\mbf i,\mbf a,\lambda}$ of $\{\mak a \in \widetilde{\Xi}_{n,\mak d}^{ap}\mid  |\mak a|=d\}$ such that
\begin{align*}
_{d+pn}\mcal D_{\mbf i,\mbf a,\lambda+2pI}=\sum_{\mak a \in \mcal I_{\mbf i,\mbf a,\lambda}}c_\mak a\zeta_{_{2p}\mak a}, \ \forall p, \  {\rm where} \ c_{\mak a} \in \mcal A \ {\rm is \ independent \ of} \ p.
\end{align*}
\end{proposition}

 Following the argument of \cite[Proposition 6.2.4]{FLLLW20}, which is formal and not reproduced here, we have the following proposition.
\begin{proposition}
The set $\{\zeta_{\widehat{\mak a}}\mid \widehat{\mak a} \in \widetilde{\Xi}_{n,\mak d}^{ap}/\approx\}$ forms a basis for $\dot{\mbf U}_n^\mak d$ and an $\mcal A$-basis for $\dot{\mbf U}_{n;\mcal A}^\mak d$.
\end{proposition}

\subsection{Bilinear form on $\dot{\mbf U}_n^\mak d$}

Imitating MaGerty \cite{Mc12} in affine type $\mrm A$,
we define a bilinear form $\langle \cdot,\cdot \rangle_d$ on $\mbf S_{n,d}^\mak d$ as follows
\begin{align*}
\langle [\mak a],[\mak a']\rangle_d=\delta_{[\mak a],[\mak a']}
v^{-2d(\mak a^t)}\sharp X_{\mak a^t}^{L'},
\end{align*}
where $L' \in \mcal X_{n,d}^\mak d({\rm ro}(\mak a^t))$.
With the help of the identity \eqref{Dimen eq},
the same argument as in \cite[Proposition 3.2]{Mc12} gives us the following.
\begin{proposition} \label{adjoint}
We have $\langle[\mak a]\ast[\mak b],[\mak c]\rangle_d=
\langle[\mak b],v^{d(\mak a)-
d(\mak a^t)}[\mak a^t]\ast[\mak c]\rangle_d$.
\end{proposition}

\begin{corollary}
For all $i \in [1,r-1]$, we have the following:
\begin{align*}
\langle \mbf E_i[\mak a_1], [\mak a_2]\rangle_d&=\langle[\mak a_1], v\mbf K_i\mbf F_i[\mak a_2]\rangle_d,  & \ \ \langle \mbf F_i[\mak a_1], [\mak a_2]\rangle_d&=\langle[\mak a_1], v^{-1}\mbf E_i\mbf K_i^{-1}[\mak a_2]\rangle_d, \\
\langle \mbf T_0[\mak a_1], [\mak a_2]\rangle_d&=\langle[\mak a_1], \mbf T_0[\mak a_2]\rangle_d, & \ \  \langle \mbf T_r[\mak a_1], [\mak a_2]\rangle_d&=\langle[\mak a_1], \mbf T_r[\mak a_2]\rangle_d,  \\
\langle \mbf J_{\pm}[\mak a_1], [\mak a_2]\rangle_d&=\langle[\mak a_1], \mbf J_\pm[\mak a_2]\rangle_d, &  \ \  \langle \mbf K_i[\mak a_1], [\mak a_2]\rangle_d&=\langle[\mak a_1], \mbf K_i[\mak a_2]\rangle_d, & \ &\
\end{align*}
\end{corollary}
\begin{proof}
The cases for $\mbf K_i$ and $\mbf J_\pm$ is clear.

If $\mak a=(A,\alpha)$ and $A-E_{\theta}^{i+1,i}$ is diagonal for some $i\in [1,r-1]$.
Thus, we have
\begin{align*}
\mak a^t=(A^t, \alpha), \ d(\mak a)={\rm co}(\mak a)_{i+1} \ \  {\rm and} \ \ d(\mak a^t)={\rm ro}(\mak a)_i={\rm co}(\mak a)_i-1.
\end{align*}
Hence, $d(\mak a)-d(\mak a^t)={\rm co}(\mak a)_{i+1}-{\rm co}(\mak a)_i+1$.
Moreover, we have
\begin{align*}
v\mbf K_i[\mak a^t]\ast[\mak a_2]=v^{1+{\rm co}(\mak a)_{i+1}-{\rm co}(\mak a)_i}
[\mak a^t]\ast[\mak a_2],
\end{align*}
which implies the case for $\mbf E_i$.

We now prove the case  $\mbf F_i$.
If $\mak a=(A,\alpha)$ and $A-E_{\theta}^{i,i+1}$ is diagonal for some $i\in [1,r-1]$.
Thus, we have
\begin{align*}
\mak a^t=(A^t, \alpha), \ d(\mak a)={\rm co}(\mak a)_i \ \  {\rm and} \ \ d(\mak a^t)={\rm ro}(\mak a)_{i+1}={\rm co}(\mak a)_{i+1}-1.
\end{align*}
So $d(\mak a)-d(\mak a^t)={\rm co}(\mak a)_{i}-{\rm co}(\mak a)_{i+1}-1$.
We further have
\begin{align*}
v\mbf K_i^{-1}[\mak a^t]\ast[\mak a_2]=v^{1+{\rm co}(\mak a)_i-{\rm co}(\mak a_{i+1})}[\mak a^t]\ast[\mak a_2].
\end{align*}
Then this case follows since $\mbf E_i\mbf K_i=v^2\mbf K_i\mbf E_i$.

As for the case $\mbf T_0$, If $\mak a=(A,\alpha)$ and $A-E_{\theta}^{0,1}$ is diagonal.
Thus we have
\begin{align*}
\mak a^t=(A, \alpha'), \alpha'\neq \alpha, \ {\rm and} \ d(\mak a)=d(\mak a^t)
={\rm co}(\mak a)_1-1.
\end{align*}
This implies the case of $\mbf T_0$.
The proof of case $\mbf T_r$ is similar, and we omit it.
\end{proof}
The same argument as in \cite{Mc12} shows that there is a well-defined bilinear $\langle\cdot,\cdot\rangle$ on $\dot{\mbf U}_n^\mak d$ given by
\begin{align*}
\langle x ,y\rangle=\sum_{d=1}^n{\underset{p\rightarrow\infty}{\rm lim}}\langle x_{d+pr}, y_{d+pn}\rangle_{d+pn}, \ \   \forall x=(x_d),y=(y_d) \in \dot{\mbf U}_n^\mak d.
\end{align*}
\begin{remark}
The same adjointness property as in Proposition \ref{adjoint} holds for the bilinear form $\langle\cdot,\cdot\rangle$ on $\dot{\mbf U}_n^\mak d$.
\end{remark}

Following the argument of \cite[Proposition 5.2]{LW15} and \cite{Mc12},
which is formal and not reproduced here, we have the following proposition.
\begin{proposition}\label{monomial}
For any $A \in \Xi_\mak d^{ap}$, we have
\begin{align*}
\phi_{d+pn,d+(p-1)n}^\mak d(\{_{2p}\mak a\}_{d+pn})
=\{_{2p-2}\mak a\}_{d+(p-1)n}, \ \  \forall p\gg 0.
\end{align*}
Moreover, we have
\begin{align*}
\{_{2p}\mak a\}_{d+pn}=\zeta_{_{2p}\mak a}+\sum_{\mak b\in \widetilde{\Xi}_{n,\mak d}^{ap}; \mak b<\mak a}c_{\mak a,\mak b,p}\zeta_{_{2p}\mak b}
\end{align*}
with $c_{\mak a,\mak b,p} \in \mcal A$ independent of $p$ for $p\gg0$.
\end{proposition}

\begin{definition}
For any $\widehat{\mak a} \in \Xi_{n,\mak d}^{ap}/\approx$, an element $b_{\widehat{\mak a}} \in \dot{\mbf U}_n^\mak d$ is defined as follows: $(b_{\widehat{\mak a}})_d=0$ if $d \neq |\mak a| \ {\rm mod} \ 2n$. If $|\mak a|=d+sn$ for some integer $s$, we set
\begin{align*}
(b_{\widehat{\mak a}})_{d+sn+pn}=\{_{2p}\mak a\}_{d+sn+pn},  \ \  \forall p\geq p_0, \  {\rm for \ some \ fixed} \ p_0,
\end{align*}
and for general $p<p_0$,
we set $(b_{\widehat{\mak a}})_{d+sn+pn}=\phi_{d+sn+p_0n,d+sn+pn}
(\{_{2p_0}\mak a\}_{d+sn+p_0n}).$
\end{definition}

The fact that $b_{\widehat{\mak a}}$ as defined above lies in $\dot{\mbf U}_n^\mak d$ follows from Proposition \ref{monomial}.
Moreover, $\zeta_{\widehat{\mak a}}=b_{\widehat{\mak a}}+{\rm lower \ terms}$.
The next theorem follows from that $\{\zeta_{\widehat{\mak a}} \mid \widehat{\mak a} \in \widetilde{\Xi}_{n,\mak d}^{ap}/\approx\}$ forms a basis for $\dot{\mbf U}_n^\mak d$.
\begin{theorem}\label{1500}
The set $\dot{\mbf B}_n^\mak d:=\{b_{\widehat{\mak a}}\mid \widehat{\mak a} \in \widetilde{\Xi}_{n,\mak d}^{ap}/\approx\}$ forms a basis for $\dot{\mbf U}_n^\mak d$.
\end{theorem}

The basis $\dot{\mbf B}_n^\mak d$ is called the canonical basis of $\dot{\mbf U}_n^\mak d$.

\begin{proposition}
The signed canonical basis $\{\pm b_{\widehat{\mak a}} \mid \widehat{\mak a} \in \widetilde{\Xi}_{n,\mak d}^{ap}/\approx\}$ is characterized by the bar-invariance, integrality (i.e., $b_{\widehat{\mak a}}\in \dot{\mbf U}_{n;\mcal A}^\mak d$), and almost orthonormality (i.e., $\langle b_{\widehat{\mak a}},b_{\widehat{\mak a'}}\rangle=\delta_{\mak a,\mak a'} \ {\rm mod} \ v^{-1}\mbb Z[[v^{-1}]]$).
\end{proposition}

The canonical basis of $\dot{\mbf U}_n^\mak d$
enjoy several remarkable positivity properties as follows.
 The proof use the same argument as in \cite{LW15,FL14,FL15}.

\begin{theorem}\label{1501}
The structure constants of the canonical basis $\dot{\mbf B}_n^\mak d$
 lie in $\mbb N[v,v^{-1}]$ with
respect to the multiplication, and in $v^{-1}\mbb N[[v^{-1}]]$ with respect to the bilinear pairing.
\end{theorem}

\subsection{The quantum symmetric pair}
 Recall from Proposition \ref{1401}, there is an injective algebra homomorphism $\jmath_{n,d}: \widetilde{\mbf U}_{n,d}^\mak d\rightarrow \mbf U_{n,d}$, and we have the following commutative diagram
 \begin{align*}
  \xymatrix{
  \widetilde{\mbf U}_{n,d}^\mak d \ar[d]_{\widetilde{\phi}_{d,d-n}^\mak d} \ar[r]^{\jmath_{n,d}}
                & \mbf U_{n,d} \ar[d]^{\phi_{d,d-n}}  \\
  {\widetilde{\mbf U}_{n,d-n}^\mak d} \ar@{} \ar[r]_{\jmath_{n,d-n}}
                & \mbf U_{n,d-n}     }
 \end{align*}
 That is, $\phi_{d,d-n}\circ\jmath_{n,d}=\jmath_{n,d-n}\circ\widetilde{\phi}_{d,d-n}^\mak d$.
 Thus by the universality of $\mbf U_{n,\infty}$ (the projective limit of Lusztig algebras of type $\mrm A$), we have a unique algebra homomorphism
 \begin{align*}
 \jmath_n: \widetilde{\mbf U}_{n,\infty}^\mak d\rightarrow \mbf U_{n,\infty},
 \end{align*}
 such that $\widetilde{\phi}_d^\mak d\circ\jmath_n=\jmath_{n,d}\circ\widetilde{\phi}_d^\mak d$.

Since $\jmath_{n,d}$ is injective for all $d$, so is $\jmath_n: \widetilde{\mbf U}_{n,\infty}^\mak d\rightarrow \mbf U_{n,\infty}$.
It follows by (\ref{chen2}) that the image of $\widetilde{\mbf U}_n^\mak d$ under $\jmath_n$ lies in $\mbf U_n$, where $\mbf U_n$ denote the subalgebra generated by Chevalley generators $\mbf e_i, \mbf f_i$ and $\mbf k_i$ in the projective limit $\mbf U_{n,\infty}$.
Summarizing, we have obtained the following.
\begin{proposition}\label{1400}
There is a unique algebra imbedding $\jmath_n: \widetilde{\mbf U}_n^\mak d\rightarrow\mbf U_n$ such that
\begin{align}\label{chen1215}
\jmath_n(\mbf E_i)&=\mbf e_i+\mbf k_i\mbf f_{n-i},&
\jmath_n(\mbf F_i)&= \mbf e_{n-i}+\mbf f_i\mbf k_{n-i},\nonumber\\
\jmath_n(\mbf T_0)&=\mbf e_0+v^{-1}\mbf f_0\mbf k_0,&
\jmath_n(\mbf T_r)&=\mbf e_r+v^{-1}\mbf f_r\mbf k_r,\\
\jmath_n(\mbf K_i)&=\mbf k_i\mbf k_{n-i}. &\ &\ \nonumber
\end{align}
\end{proposition}

Recall $\Delta^\mak d$ from (\ref{1212}), we have the following commutative diagram
\begin{align*}
\xymatrix{
  \widetilde{\mbf U}_{n,d'+d''}^\mak d \ar[d]_{\widetilde{\phi}_{d'+d'',d'+d''-(a+b)n}^\mak d} \ar[r]^{\Delta^\mak d}
                & \widetilde{\mbf U}_{n,d'}^\mak d\otimes \mbf U_{n,d''} \ar[d]^{\widetilde{\phi}_{d',d'-an}^\mak d
                \otimes\phi_{d'',d''-bn}}  \\
  \widetilde{\mbf U}_{n,d'+d''-(a+b)n}^\mak d \ar@{} \ar[r]_{\Delta^\mak d}
                & \widetilde{\mbf U}_{n,d'-an}^\mak d\otimes \mbf U_{n,d''-bn}}
\end{align*}
for any $a,b \in \mbb N$.
By universality, these $\Delta^\mak d \ ({\rm for} \ d',d'',n)$ induce an algebra homomorphism
\begin{equation*}
  \Delta^\mak d: \widetilde{\mbf U}_{n,\infty}^\mak d\rightarrow
  \widetilde{\mbf U}_{n,\infty}^\mak d\otimes \mbf U_{n,\infty}.
\end{equation*}
Moreover, the image of $\widetilde{\mbf U}_n^\mak d$ under $\Delta^\mak d$ is contained in $\widetilde{\mbf U}_n^\mak d\otimes \mbf U_n$ by Proposition \ref{1216}.
To sum up, we have the following.
\begin{proposition}\label{1900}
There is a unique algebra homomorphism $\Delta^\mak d: \widetilde{\mbf U}_n^\mak d\rightarrow
\widetilde{\mbf U}_n^\mak d\otimes \mbf U_n$ such that, for all $i\in [1,r-1]$,
\begin{align*}
\Delta^\mak d(\mbf E_i) &= \mbf E_i\otimes \mbf k_i+1\otimes\mbf e_i+\mbf K_i\otimes\mbf f_{n-i}\mbf k_i. \\
\Delta^\mak d(\mbf F_i) &= \mbf F_i\otimes \mbf k_{n-i}+\mbf K_i^{-1}\otimes \mbf k_{n-i}\mbf f_i+1\otimes\mbf e_{n-i}.\\
\Delta^\mak d(\mbf T_0) &= \mbf T_0\otimes\mbf k_0+1\otimes \mbf e_0+v^{-1}\otimes\mbf f_0\mbf k_0.\\
\Delta^\mak d(\mbf T_r) &= \mbf T_r\otimes \mbf k_r+1\otimes\mbf e_r+v^{-1}\otimes\mbf f_r\mbf k_0.\\
\Delta^\mak d(\mbf K_i) &= \mbf K_i\otimes\mbf k_i\mbf k_{n-i}^{-1}.
\end{align*}
\end{proposition}

This algebra homomorphism is coassociative by Proposition \ref{1217} in the sense that
\begin{equation}\label{1218}
 (1\otimes \Delta)\Delta^\mak d=(\Delta^\mak d\otimes 1)\Delta^\mak d.
\end{equation}
As a degenerate case for (\ref{1218})
\begin{align*}
\Delta\circ\jmath_n=(\jmath_n\otimes 1)\circ \Delta^\mak d.
\end{align*}
The algebra $\widetilde{\mbf U}_n^\mak d$ is a coideal subalgebra of $\mbf U_n$,
and the pair $(\mbf U_n, \widetilde{\mbf U}_n^\mak d)$ forms an affine quantum symmetric pair in the sense of Kolb-Letzter \cite{Ko14}.
The relevant involution is illustrated as follows:
\begin{figure}[ht!]
   \label{figure:ii}
   \begin{center}
\begin{tikzpicture}
\matrix [column sep={0.6cm}, row sep={0.5 cm,between origins}, nodes={draw = none,  inner sep = 3pt}]
{
	&\node(U1) [draw, circle, fill=white, scale=0.6, label = 1] {};
	&\node(U2) {$\cdots$};
	&\node(U3)[draw, circle, fill=white, scale=0.6, label =$r-1$] {};
\\
	\node(L)[draw, circle, fill=white, scale=0.6, label =0] {};
	&&&&
	\node(R)[draw, circle, fill=white, scale=0.6, label =$r$] {};
\\
	&\node(L1) [draw, circle, fill=white, scale=0.6, label =below:$2r-1$] {};
	&\node(L2) {$\cdots$};
	&\node(L3)[draw, circle, fill=white, scale=0.6, label =below:$r+1$] {};
\\
};
\begin{scope}
\draw (L) -- node  {} (U1);
\draw (U1) -- node  {} (U2);
\draw (U2) -- node  {} (U3);
\draw (U3) -- node  {} (R);
\draw (L) -- node  {} (L1);
\draw (L1) -- node  {} (L2);
\draw (L2) -- node  {} (L3);
\draw (L3) -- node  {} (R);
\draw (L) edge [color = blue, loop left, looseness=40, <->, shorten >=4pt, shorten <=4pt] node {} (L);
\draw (R) edge [color = blue,loop right, looseness=40, <->, shorten >=4pt, shorten <=4pt] node {} (R);
\draw (L1) edge [color = blue,<->, bend right, shorten >=4pt, shorten <=4pt] node  {} (U1);
\draw (L3) edge [color = blue,<->, bend left, shorten >=4pt, shorten <=4pt] node  {} (U3);
\end{scope}
\end{tikzpicture}
\end{center}
\caption{Dynkin diagram of type $\mrm A^{(1)}_{2r-1}$ with involution of type.}
\end{figure}
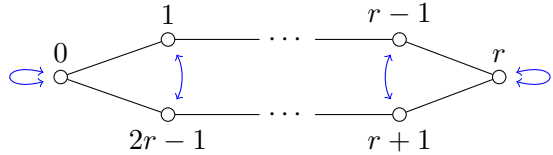

\end{document}